\documentclass[leqno,11pt,a4paper]{amsart}

\usepackage{amssymb,latexsym}
\usepackage{graphicx}
\usepackage{hyperref}
\usepackage{amsmath}
\usepackage{amsfonts}
\usepackage{amssymb}
\usepackage{float}%
\usepackage{tikz-cd}

\usepackage{bbm}

\usepackage{nicefrac}

\usepackage[new]{old-arrows}

\theoremstyle{plain} 
\newtheorem{theorem}{\bf Theorem}[section]

\newtheorem{lemma}[theorem]{\bf Lemma}
\newtheorem{corollary}[theorem]{\bf Corollary}

\newtheorem{proposition}[theorem]{\bf Proposition}

\theoremstyle{definition} 
\newtheorem{definition}[theorem]{\bf Definition}
\newtheorem{remark}[theorem]{\bf Remark}
\newtheorem{question}[theorem]{\bf Question}
\newtheorem{example}[theorem]{\bf Example}

\newtheorem{conj}[theorem]{\bf  Conjecture}
\newcommand*{\DashedArrow}[1][]{\mathbin{\tikz [baseline=-0.25ex,-latex, dashed,#1] \draw [#1] (0pt,0.5ex) -- (1.3em,0.5ex);}}%

\newcommand{\res}{\operatorname{res}}

\newcommand{\Spec}{\operatorname{Spec}}

\newcommand{\Socle}{\operatorname{Socle}}

\makeatletter
\DeclareRobustCommand{\Udots}{%
  \vcenter{\offinterlineskip
    \halign{%
      \hbox to .8em{##}\cr
      \hfil.\cr\noalign{\kern.2ex}
      \hfil.\hfil\cr\noalign{\kern.2ex}
      .\hfil\cr}%
  }%
}
\makeatother

 \newcommand{\R}{\mathbb{R}}

\newcommand{\sbseq}{\subseteq}
\newcommand{\spseq}{\supseteq}

\newcommand{\vanish}[1]{}

 \newcommand{\de}{\delta}

\def\sbs\subset
\def\sbseq{\subseteq}

\def\langle{\left<}
\def\rangle{\right>}

\def\({\left(}     
\def\){\right)}
\def\no={\,{\,|\!\!\!\!\!=\,\,}}

\def\no={\,{\,|\!\!\!\!\!=\,\,}}
\def\sbseq{\subseteq}

\def\sgn{\operatorname{sgn}}

\def\sbseq{\subseteq}
\def\sbs\subset
\def\spseq{\supseteq}

\newcommand{\xqedhere}[2]{%
	\rlap{\hbox to#1{\hfil\llap{\ensuremath{#2}}}}}

\newcommand\Defn[1]{\textbf{#1}}
\newcommand{\cm}[1]{}

\newcommand\mbf[1]{\mathbf{#1}}

\newcommand\mr[1]{\mathrm{#1}}

\newcommand{\fld}{\mathbbm{k}}

\newcommand\vo{\mr{vol}}
\renewcommand\de{\mr{deg}}
\newcommand{\bigslant}[2]{{\raisebox{.3em}{$#1$} \Big/ \raisebox{-.3em}{$#2$}}}

\newcommand\x{\mathbf{x}}

\newcommand\AR{\mathcal{A}}

\title
{The volume intrinsic to a commutative graded algebra}

\author[Karim Adiprasito]{Karim Alexander Adiprasito}
\address{{Karim Adiprasito}, Sorbonne Université and Université Paris Cité, CNRS, IMJ-PRG, F-75005 Paris, France}

\email{karim.adiprasito@imj-prg.fr }

\author[Stavros Papadakis]{Stavros Argyrios Papadakis}
\address{{Stavros Papadakis}, Department of Mathematics, University of Ioannina, Ioannina, 45110, Greece}
\email{spapadak@uoi.gr}

\author[Vasiliki Petrotou]{Vasiliki Petrotou}
\address{{Vasiliki Petrotou}, Einstein Institute of Mathematics, Hebrew University
of Jerusalem,  Jerusalem,  91904 Israel \ \emph{ and} \  Sorbonne Université and Université Paris Cité, CNRS, IMJ-PRG, F-75005 Paris, France}
\email{vpetrotou1@gmail.com}

\begin{document}


\begin{abstract}
Recent works of the authors have demonstrated the usefulness of considering moduli spaces of Artinian reductions of a given ring when studying 
standard graded rings and their Lefschetz properties. This paper illuminates a key aspect of these works, the 
behaviour of the canonical module under deformations in this moduli space. We demonstrate that even when there 
is no natural geometry around, we can give a viewpoint that behaves like it, effectively constructing geometry 
out of nothing, giving interpretation to intersection numbers without cycles. Moreover, we explore some properties of this normalization.

\end{abstract}

\maketitle

\section{Introduction}   

The fundamental class of a variety or manifold is an important invariant, 
not least of all because it allows us to formulate Poincar\'e  Duality. 
And it has a natural geometric interpretation as well: in a $d$-dimensional smooth 
manifold, de Rham cohomology associates it with the integration of a $d$-form over the manifold. And 
at this point, we can study what happens to the fundamental class under small deformations of the 
manifold. And if it makes sense to fix a $d$-form $\omega$ under this deformation, how does the evaluation under the fundamental class change?

We are in particular inspired by another case: That of toric varieties. If we consider a complete, rationally smooth toric 
variety of dimension $d$, we can consider it as part of an entire moduli space: we consider the family of all toric varieties 
with the same underlying combinatorics, or in other words, the same equivariant cohomology, parametrized by the 
torus action. Say it makes sense to consider an element $\omega$ of the top cohomology in such a way that we can study it under 
a deformation of the torus action. How does the evaluation change when we vary in the moduli space?

This question is invaluable, in particular in recent works of the authors, and studying it led to 
advances in combinatorial Lefschetz theory. It is, moreover, a first step towards understanding these new theories.

Let us once more examine the case of toric varieties, and explain what we mean.

Consider a unimodular complete fan $\Sigma$ in $\R^d$, and the toric variety $X$ associated 
to it. There are natural ways of thinking about the cohomology ring of $X_\Sigma$.

Here are two: 

\textbf{First}, we can consider the ring of conewise polynomial functions $\mathcal{P}(\Sigma)$; it is isomorphic to the equivariant cohomology ring of $X_\Sigma$.

The cohomology ring of $X_\Sigma$ is equally easy to compute:
\[\mathcal{A}(\Sigma)\ :=\ \bigslant{\mathcal{P}(\Sigma)}{\langle \mathcal{G}\rangle}\]
where $\mathcal{G}$ is the ideal generated by global polynomials. It is naturally a Poincar\'e duality algebra, so that 
\[\mathcal{A}(\Sigma)_d\ \cong\ \R.\]
But what is the map? What is
\[\vo:\mathcal{A}(\Sigma)_d\ \rightarrow\ \R.\]
It is nothing more natural than to think about this map, also known as the volume map, in terms of de Rham cohomology. We then 
obtain, assuming the fan $\Sigma$ is unimodular, or equivalently $X_\Sigma$ is smooth, a natural normalization. Consider for instance the natural basis for
$\mathcal{A}(\Sigma)_1$: the characteristic function of a ray $\chi_\rho$ is conewise linear and vanishes on each 
ray of $\Sigma$ but $\rho$, where it is positive. Let $v_\rho$ denote the point on $\rho$ where it attains the value $1$.

Then, if $(\rho_{i})_{i\in [d]}$ is a collection of $d$ rays of $\Sigma$, then 
\[
       \vo\left(\prod_{i\in d} \chi_{\rho_i}\right)\ =\ \left|\frac{1 }{\Delta{v_{\rho_i}}} \right|,
\]
where  ${\Delta{v_{\rho_i}}}$  denotes the determinant of the $d \times d$ matrix with $i$-th row equal to $v_{\rho_i}$.

That is, in a circumvent way, what happens in toric varieties, where the above recovers the natural degree map coming from algebraic geometry, computing intersection numbers faithfully. As most of what we consider takes place in the realm of commutative algebra, and degree has two possible meanings, we shall forgo calling it degree throughout, and instead refer to this as the volume map throughout.

\textbf{Second}, to explain this another way, let us switch models, and discuss what happens in Stanley-Reisner rings, or rather, face rings.

Consider $\Sigma$ a triangulated sphere of dimension $d-1$, and $\fld$ an infinite field. 
The face ring of $\Sigma$ is $\fld[\Sigma]$, the quotient of the polynomial ring $\fld[\mbf{x}]$ (with indeterminates 
indexed by the vertices) by the ideal $I_\Sigma$ generated by nonfaces.

To obtain a Poincar\'e duality algebra, we consider the quotient of $\fld[\Sigma]$ by $\Theta\fld[\Sigma]$, where $\Theta=(\theta_{i,j}) \mbf{x}$
 is a collection of $d$ linear forms, encoded by a matrix $(\theta_{i,j})$ of entries in $\fld$ whose columns are indexed 
by the vertices of $\Sigma$ and the indeterminates of the face ring. If we think of the fundamental class of the simplicial homology 
as a cycle $\mu=\mu_\Sigma$, then  for $F$ a facet of $\Sigma$  we have 
\[\vo\left(\prod_{i\in F} x_i\right)\ =\ \frac{\sgn_\mu F}{\Delta ((\theta_{i,j}))_{|F}} \]
where $\sgn_\mu F$ is the oriented sign of $F$ and $\Delta ((\theta_{i,j}))_{|F}$ is the determinant of the minor of $(\theta_{i,j})$ cut out by $F$.

So, we have an, in several ways natural, normalization of the volume map. And we have established several times now how powerful understanding this function can be. Let us recall why, and then aim at understanding the volume map in general.

\subsection{Motivation: Anisotropy and generic Artinian reductions}

One of the most important motivations towards understanding the geometry of volume maps arises 
in Lefschetz theory: In \cite{AHL}, it was observed that the hard Lefschetz property for a Gorenstein standard graded ring $R$ is true 
if for a sufficiently large set of ideals $I$ in $\AR(R)$, an Artinian reduction of $R$, the Poincar\'e pairing is non-degenerate when 
restricted to $I$. In \cite{PP}, and subsequently \cite{APP, APPS}, one studies the most extreme form of this principle: non-degeneracy at principal ideals.

In other words, we consider $u$ in $\AR(R)$ of degree at most $k$, $k$ less 
or equal to $\nicefrac{d}{2}$, $d$ being the socle degree, and investigate nonvanishing 
of $u^2$. We may assume that $k=\nicefrac{d}{2}$. Now, there is an issue here: How do we choose the appropriate 
Artinian reduction, indexed by the linear system of parameters $\Theta=(\theta_{i,j}) \mbf{x}$? The trick is to pass from $R$,
which is  a quotient of a polynomial ring over a field $\fld$, to a quotient of a polynomial ring over the purely transcendental
field extension $\fld(\theta_{i,j})$.

Now, the idea is simple: We want to understand $\vo(u^2)$ as a rational function  in the variables $(\theta_{i,j})$. 

In \cite{PP} and \cite{APP}, this is then understood by considering differential identities for $\vo(u^2)$ to understand the case of face rings.

In \cite{APPS}, we explore instead semigroup algebras associated to lattice polytopes, and understand $\vo(u^2)$ by exploring Parseval-Rayleigh identities for $\vo(\cdot)$. 

But that leaves out the main question: How does $\vo(\cdot)$ behave at all, how does it depend on $\theta_{i,j}$? Our ultimate goal is to understand this problem towards the following conjecture:

\begin{conj}
	Consider $R$ a reduced standard graded Gorenstein ring. Then a generic Artinian reduction of $R$ has the strong Lefschetz property.
\end{conj}

This is motivated by a slightly less general conjecture of Stanley \cite{Braun}: he asks whether the same is true for 
integral domains $R$ (which is obviously more restrictive than reduced). Alas, to start with the methods of generic Lefschetz theory and attack 
either conjecture, the first question one must ask, and which we tackle here, is the question: What is $\vo(\cdot)$? To learn to speak of it concretely is a main goal of this paper.

\subsection{Plan for the paper} 

Philosophically speaking, our normalization is the minimal polynomial identity the volume map satisfies. We make this precise 
in the next section by introducing the notion of a generic polynomial reduction. We then introduce the Kustin-Miller, or KM normalization, 
using this notion, and immediately note two alternative ways of doing so. We then illustrate our theory on some important examples. 
To make the theory more flexible, we finally introduce relative normalization, and show that it explains several phenomena 
of the normalization more succinctly, and allows us to explain the normalization for semigroup algebras of lattice polytopes 
introduced and studied in \cite {APPS}.

\section{Generic polynomial reduction}      \label{sec!gericpolynomialreduction}

We introduce and study the notion of generic polynomial reduction of a graded algebra, 
which will play an essential role
in the definition of volume normalization in Section~\ref{sec!dfnofdegreenormalizationiso}.

Assume $m \geq 1$ and  $\fld$ is a field.  We consider the polynomial 
ring  $\fld[x_1, \dots , x_m]$, where the degree of the variable $x_i$ is equal to $1$,
for all $1 \leq i \leq m$.  Assume   $I \subset \fld[x_1, \dots , x_m]$ 
is a homogeneous ideal. We denote by $d$ the  
Krull dimension of the quotient ring  $\fld[x_1, \dots , x_m]/I$ 
and  assume  $d \geq 1$.

We assume $\theta_{i,j}$ are new variables, and we set 
$R_{up}$ to be the  graded polynomial ring
\[
R_{up} =    \fld[x_1, \dots , x_m, \theta_{i,j} :  1 \leq i \leq d,  \;  1 \leq j \leq m ].
\]
We equip $R_{up}$ with a bidegree as follows:  We say that an element  $u \in R_{up}$
is  bihomogeneous of degree $(b,c)$ if when considered as a polynomial in $x_i$ 
is homogeneous of degree $b$ and when considered as a polynomial in $\theta_{i,j} $ 
is homogeneous of degree $c$.
We set,  for $1 \leq i \leq d$,
\[
f_i = \sum_{j=1}^{m} \theta_{i,j} x_j.
\]
We denote by $I_{up}$ the ideal   of $R_{up}$ generated by the subset 
$\; I \cup \{ f_1, \dots , f_d\}$,  and by $N_{up}$ the ideal
\[
N_{up} =  I_{up}  : (x_1, \dots , x_m).
\]

\begin{definition}  \label{dfn!genericpolynomialreduction}
	We define the {\it generic polynomial reduction}  of   $\fld[x_1, \dots , x_m]/I$  to be the
	$R_{up}$-algebra
	\[
	R_{up}/ I_{up}.
	\]
\end{definition}

In Subsection~\ref{subs!dimensionformulaforrupjup} we will prove the following important
proposition.

\begin{proposition}   \label{prop!dimformulaforrupjup}  We have 
	\[
	\dim  ( R_{up} / I_{up})  =  dm.
	\]    
	In other words, the Krull dimension of $  R_{up} / I_{up}$ is equal to $dm$.
\end{proposition}

It is clear that the ideals $I_{up}$ and $N_{up}$ are bihomogeneous, hence it follows that
$N_{up}/ I_{up}$  is a  bihomogeneous ideal of the bigraded ring $R_{up}/ I_{up}$.  
For $b \geq 0$  we set
\[
(N_{up}/ I_{up})_{(b,-)} = \bigoplus_{c \geq 0}(N_{up}/ I_{up})_{(b,c)}.
\]
Since $\;  x_i (N_{up}/ I_{up}) = 0$ for all $1 \leq i \leq m$, it follows that 
$(N_{up}/ I_{up})_{(b,-)}$  is an $R_{up}$-submodule of  $N_{up}/ I_{up}$.

In the following,   $h_1,  \dots , h_s$  denote $s$ bihomogeneous elements of $R_{up}$
with the property that $h_1+I_{up},  \dots , h_s+I_{up}$ 
is a  minimal generating set for the $R_{up}$-module $N_{up}/I_{up}$.  Moreover,
$\de_{x} h_t$ denotes  the degree of $h_t$ with respect to the variables $x_j$.

We denote by    $\fld[\theta_{i,j}]$  the polynomial ring
\[
\fld[ \theta_{i,j} :  1 \leq i \leq d,  \;  1 \leq j \leq m ],
\]
by  $E=\fld(\theta_{i,j})$ the field of fractions  of $\fld[\theta_{i,j}]$,  and we set 
\[
R_{down} =    E [x_1, \dots , x_m ],
\]
with $\de(x_i) = 1$ for all $i$.

We denote by $I_{down}$ the ideal   of $R_{down}$ generated by the subset 
$\; I \cup \{ f_1, \dots , f_d\}$,  and by $N_{down}$ the ideal
\[
N_{down} =    I_{down}  : (x_1, \dots , x_m).
\]

Remember now that the 
{\it generic Artinian reduction}  $\AR$  of  $\fld[x_1, \dots , x_m]/I$ is the
Artinian $E$-algebra
\[
\AR = R_{down}/I_{down}.
\]
We denote by $e$ the top degree of  $\AR$, this means that the $e$-th graded
component $\AR_e$ of  $\AR$  is nonzero and $\AR_s = 0$  when $s \geq e+1$.  Obviously  
$\AR_e \subset   \Socle(\AR)$, where 
\[
\Socle(\AR)  = \{  u \in \AR : \; \; u x_i = 0   \;  \; \text{for all}  \;  \; 1 \leq i \leq m  \}
\]
denotes the socle of  $\AR$ .  Moreover, from the definition of $N_{down}$ it follows that
\[
N_{down}/I_{down} =   \Socle(\AR).
\]

We set 
\[
S = \fld[\theta_{i,j} :  1 \leq i \leq d,  \;  1 \leq j \leq m ]  \setminus \{  0 \}.
\]
Clearly  $S$ is a multiplicatively closed subset of $R_{up}$.  Every element of $R_{down}$ 
can be written in  the form  $p/u$   with $p \in  R_{up}$ and $u \in S$. It follows that
$R_{down}$ and $S^{-1}R_{up}$ are naturally isomorphic. In particular, there exists a
natural injective localization map  $\Phi_1 : R_{up} \to  R_{down}$  that sends $p$ to $p$
for all $p \in R_{up}$. 

It is clear that  $\Phi_1  (I_{up}) \subset  I_{down}$ and
$\Phi_1  (N_{up}) \subset  N_{down}$. Consequently, there is an induced homomorphism
\[
\Phi :  N_{up} / I_{up}  \to    N_{down} / I_{down}  = \Socle(\AR) 
\] 
such that   $\; \Phi (u+ I_{up}) =  u+ I_{down} \;$  for all $u \in  N_{up}$.
The map $\Phi$ respects degrees, in  the sense that 
\[
\Phi ( (N_{up} / I_{up})_{(b,-)})   \subset    (N_{down} / I_{down})_b 
\] 
for all $b \geq 0$.   

\begin{proposition}   \label{proposmapessentiallysurjective}  We have 
	\[
	(  \Phi ( N_{up} / I_{up}))  = \Socle(\AR) .
	\]    
	In other words, the ideal of $\AR$  generated by the image of $\Phi$ is equal to 
	$ \Socle(\AR) $.   
\end{proposition}

\begin{proof}   Since $ I_{down} = S^{-1} I_{up} $,
	the result follows from  \cite[p.~42, Remark~1]{AM}.
\end{proof}

\begin{corollary}   \label{corol!relation758753} 
	Assume $i \geq 0$.   Then 
	\[
	\{  \; \Phi ( h_t +  I_{up})  \;  :  \;   \de_{x} h_t =  i  \;  \}
	\]
	is a generating set for the $E$-vector space $(\Socle(\AR))_i$.
\end{corollary}

\begin{proof}  Since for all $1 \leq j \leq m$ we have  $x_j (h_t +  I_{up}) = 0 + I_{up}$,
	we have that  the result follows   from  Proposition~\ref{proposmapessentiallysurjective}.  
\end{proof}

\begin{remark}   \label{remew!relation758753} 
	An immediate consequence of  Corollary~\ref{corol!relation758753} is that if $i \geq 0$ has the
	property that  $(\Socle(\AR))_i$ is nonzero, then  there exists $t$ such that   $\de_{x} h_t = i$.
	In particular, since  $ (\Socle(\AR))_e = \AR_e  \not= 0$, there exists $h_t$  with  $\de_{x} h_t = e$.
\end{remark}

\begin{remark}   \label{rem!odlosalf}
	See  Examples~\ref{example!intwovariables} and 
	\ref{example!inthreevariables}  for two examples where $\Phi$ is not injective
	and  (for suitable $i$)  the set 
	\[
	\{ \;   \Phi ( h_t +  I_{up})   \;  :  \;   \de_{x} h_t =  i   \;  \}
	\]
	is not a  basis of $\AR_i$ as  $E$-vector space.
\end{remark}

\begin{corollary}   \label{cor!irrelevantisaminimalprime}  
	The ideal   $( x_1,  \dots , x_m )$  of $R_{up}$ is a minimal associated prime ideal of  $R_{up} / I_{up}$.
\end{corollary}

\begin{proof}   
	We have  $I_{up} \subset ( x_1,  \dots , x_m )$. Since 
	$\dim  R_{up}/( x_1,  \dots , x_m ) = dm$ and 
	by, Proposition~\ref{prop!dimformulaforrupjup},
	$\dim   R_{up} / I_{up}  =  dm$, we get that the ideal 
	$( x_1,  \dots , x_m )$ is a minimal prime of $I_{up}$. 
	Hence, by  \cite[Theorem~3.1 a.]{Eis}
	it is also an associated prime.  
\end{proof}

\begin{corollary}   \label{cor!cohenmacaulaysubcase21} Assume 
	$\fld[x_1, \dots , x_m]/I$ is  Cohen-Macaulay.  Then  $R_{up} / I_{up}$ 
	is Cohen-Macaulay.
\end{corollary}

\begin{proof}   
	Using the assumption,  we get by Proposition~\ref{prop!dimformulaforrupjup} 
	that $R_{up}/ I_{up}$ is Cohen-Macaulay, since it is the quotient of 
	$R_{up}/ (I)$ by a homogeneous regular sequence.  
\end{proof}

\begin{proposition}   \label{prop!whatifnboembeddedprimes} 
	Assume    $R_{up} / I_{up}$ has no embedded associated prime ideals.  Then we have  that the map 
	$\Phi :  N_{up} / I_{up}   \to  \Socle(\AR)$ is injective. 
\end{proposition}

\begin{proof}   
	Suppose  $u \in  N_{up} \setminus I_{up}$ has the property that $\Phi (u + I_{up}) = 0$.
	Since,  by \cite[p.~39, Corollary~3.4 iii)]{AM},  $\Phi$ is a localization map
	there exists,  by \cite[p.~37]{AM}, nonzero $p \in   \fld[\theta_{i,j}]$
	such that  $p u \in  I_{up}$.  Hence  
	\[
	(x_1,  \dots  ,x_m,  p)  \subset   I_{up} : (u).
	\]
	By \cite[Proposition~3.4]{Eis} there exists an associated prime $q$ of  $R_{up} / I_{up}$ such that
	\[
	(x_1,  \dots  ,x_m,  p)  \subset   q.
	\]
	This is a contradiction, since by
	Corollary~\ref{cor!irrelevantisaminimalprime}   $(x_1,  \dots  ,x_m)$ is
	an associated prime  of $R_{up}/ I_{up}$ and, by assumption,
	$R_{up} / I_{up}$ has no embedded associated prime ideals.
\end{proof}

\begin{remark}   \label{rem!ekdiaskf}
	See  Examples~\ref{example!intwovariables} and 
	\ref{example!inthreevariables}  for two examples where $\Phi$ is not injective.
\end{remark}

\begin{corollary}   \label{cor!cohenmacaulayinjctivity} Assume 
	$\fld[x_1, \dots , x_m]/I$ is  Cohen-Macaulay.  Then we have that  the map 
	$\Phi :  N_{up} / I_{up}  \to  \Socle(\AR)$ is injective. 
\end{corollary}

\begin{proof}   
	By   Corollary~\ref{cor!cohenmacaulaysubcase21}  $R_{up} / I_{up}$ 
	is Cohen-Macaulay, hence by  \cite[p.~58, Theorem~2.1.2]{BH}
	it has no embedded associated prime ideals.   The result follows
	from   Proposition~\ref{prop!whatifnboembeddedprimes}.
\end{proof}

Since $\fld[\theta_{i,j}]$ is a subring of  $R_{up}$,  it is natural to consider,
for $b \geq 0$,   $(N_{up}/I_{up})_{(b,-)}$ also as a $\fld[\theta_{i,j}]$-module.

\begin{proposition}   \label{prop!cycliccase} 
	Assume  that  the   $R_{up}$-module $(N_{up}/I_{up})_{(e,-)}$ is cyclic
	and denote by $u$ a generator.    Then the multiplication map 
	\[
	s_u  :   \fld[\theta_{i,j}] \to (N_{up}/I_{up})_{(e,-)}   
	\]
	with 
	\[
	s_u  ( p(\theta_{i,j}) )  =   p(\theta_{i,j}) u
	\]
	for all $p(\theta_{i,j})  \in \fld[\theta_{i,j}]$,  is an isomorphism of $\fld[\theta_{i,j}]$-modules.
\end{proposition}

\begin{proof}   
	It is clear that $s_u$ is a homomorphism of $\fld[\theta_{i,j}]$-modules. 
	Since  $x_j (N_{up}/I_{up}) = 0$ for all $1 \leq j \leq m$,  the map  $s_u$ is surjective
	by the assumption  that  $u$ generates   $(N_{up}/I_{up})_{(e,-)}$ 
	as an $R_{up}$-module.
	
	Assume    $p(\theta_{i,j})  \in \fld[\theta_{i,j}] $ is an element in the kernel $s_u$.   
	By  Corollary~\ref{corol!relation758753},  we have that
	$\AR_e$ is a $1$-dimensional vector space over $E$ 
	and  $\Phi(u)$ is an $E$-basis of $\AR_{e}$.  Hence, $s_u (p(\theta_{i,j}))= 0$
	implies  $p(\theta_{i,j}) \Phi(u) = 0$ which imples 
	$p(\theta_{i,j}) = 0$.  Therefore,  $s_u$ is also injective.
\end{proof}

\subsection{Examples and Questions}  \label{subs!examplesfornup} 

In this subsection we discuss a number of examples and  questions about
the generic polynomial reduction.   Related 	Macaulay2~\cite{GS} code
is contained in the  file \ \href{http://users.uoi.gr/spapadak/volumem2codev1.txt}
{http://users.uoi.gr/spapadak/volumem2codev1.txt}.

\begin{example}     \label{examplgorensteincase}
	Assume $\fld$ is any field  and  $\fld[x_1, \dots ,  x_n]/I$ is Gorenstein.
	Recall that  $e$ denotes  the top degree  of  the generic Artinian reduction  $\AR$ 
	of $\fld[x_1, \dots ,  x_n]/I$. We will prove  in Section~\ref{sec!gorensteinpolynomialreduction} 
	that the $R_{up}$-module $N_{up}/I_{up}$ is cyclic and 
	generated by an element   $u \in (N_{up}/I_{up})_{(e,-)}$.  
	For an explicit computation of  $u$ in the case where  $\fld[x_1, \dots ,  x_n]/I$
	is the face ring of a simplicial sphere, see  
	Section~\ref{sec!applicationtosimplicialspheres}.
\end{example}

\begin{example}     \label{example!intwovariables}
	Assume $\fld$ is any field and $I= (x_1^5,x_1x_2) \subset \fld[x_1,x_2]$.  Then, 
	Macaulay2  computations
	suggest  the following:  The generic Artinian reduction  $\AR$  of $\fld[x_1,x_2]/I$  has Hilbert function $(1,1,0,0, \dots)$,
	hence the top degree of  $\AR$  is $1$.  The $R_{up}$-module $N_{up}/I_{up}$ is minimally generated 
	by  its subset 
	\[   
	\{ \; \theta_{1,2}x_2 +I_{up} ,  \; \;   \; \;  x_1^4 + I_{up} \; \}.
	\]
	This implies that  $(N_{up}/I_{up})_{(1,-)}$ is a cyclic $R_{up}$-module  with generator 
	$\theta_{1,2}x_2 +I_{up}$ and that  $\Phi$  is not injective, since $x_1^4 + I_{up}$ is a nonzero
	element in the kernel.
\end{example}

\begin{example}     \label{example!inthreevariables}
	Assume $\fld$ is any field and $I= (x_1^5,x_1x_2,x_1x_3) \subset \fld[x_1, x_2 , x_3]$. 
	Macaulay2 computations suggest  the following:  The generic Artinian reduction 
	$\AR$ of $\fld[x_1, x_2, x_3]/I$  has Hilbert function $(1,1,0,0, \dots)$,
	hence the top degree of $\AR$ is $1$.  The $R_{up}$-module $N_{up}/I_{up}$ is minimally generated 
	by  its subset    $\{h_1+I_{up},h_2+I_{up},h_3+I_{up}\}$,   where
	\[
	h_1 =  \theta_{1,2}x_2+\theta_{1,3}x_3,  \quad    \;    h_2 =  \theta_{2,2}x_2+\theta_{2,3}x_3, 
	\quad \;  h_3 =   x_1^4.
	\]
	It holds $\theta_{2,1}(h_1+I_{up}) - \theta_{1,1}(h_2+I_{up}) = 0$  and  $\Phi$ restricted to $(N_{up}/I_{up})_{(1,-)}$ is injective.
	Therefore,    $(N_{up}/I_{up})_{(1,-)}$ is not a cyclic $R_{up}$-module 
	and  $\Phi$  is not injective, since $h_3 + I_{up}$ is a nonzero element in the kernel.
\end{example}

\begin{example}     \label{example!torustriangulation}
	Assume $\fld$ is any field and  $D$ is the simplicial complex triangulating the 
	torus $S^1 \times S^1$  described  in    \cite[p.~70, Exerc.~5.1.8]{Sch}.
	We denote by $I \subset \fld[x_1, \dots , x_{10}]$ the Stanley-Reisner ideal
	of $D$.  We remark that by Reisner's criterion
	(\cite[p.~235, Corollary~5.3.9]{BH})    $\fld[x_1, \dots , x_{10}]/I$   is not
	Cohen-Macaulay.     
	Macaulay2 computations suggest  the following:  The generic Artinian reduction 
	$\AR$ of $\fld[x_1, \dots ,x_{10}]/I$  has Hilbert function   $(1,7,13,1,0,0,  \dots)$,
	while the socle of $\AR$ has Hilbert function   $(0,0,6,1,0,0,  \dots)$.
	The $R_{up}$-module $N_{up}/I_{up}$ is minimally generated by $7$ homogeneous
	elements $u_1,  \dots , u_7$  such that  $u_i \in (N_{up}/I_{up})_{(2,-)}$ 
	for $1 \leq i \leq 6$ and   
	\[
	u_7 =   [8,9,10]  x_8x_9x_{10} + I_{up}  \;   \in    (N_{up}/I_{up})_{(3,-)}.
	\]
	Here  $[8,9,10]$  denotes the determinant of the  $3 \times 3$ submatrix 
	of the $3 \times 10$ matrix $[\theta_{i,j}]$ obtained by keeping columns $8,9$ and $10$.
	Moreover,  we have that  
	\[
	\Phi (u_1),  \dots ,   \Phi (u_7)
	\]
	is an $E$-basis of $\Socle(\AR)$.  We remark 
	that (up to sign) in  the paper  \cite{APP}  the volume normalization on $\AR$,
	which was denoted there by $\deg$,
	was defined by the condition that   $\Phi(u_7)$ maps to $1$. 
\end{example}

\begin{example}     \label{quest!examplegenerichypersurface}
	Assume that  $\fld$ is any field, $\; R= \fld[ x_1,x_2, g_{1,1}, g_{1,2},g_{2,2}] \;$
	and
	\[
	I = (g)  \subset R,
	\]
	where
	\[
	g=  g_{1,1}x_1^2 +  g_{1,2}x_1x_2 + g_{2,2}x_2^2
	\]
	is the generic degree $2$ hypersurface in $2$ variables $x_i$.
	We write  $g= p_1 x_1 + p_2 x_2$,  for some (non-unique)  $p_1,p_2 \in R$ 
	and set 
	\[
	u =   ( - p_1 \theta_{1,2} + p_2 \theta_{1,1} ) \in R_{up}.
	\]
	Then, by Subsection~\ref{subsec!completeintersectioncase}  below we have that
       $(I_{up} : (x_1,x_2)) /I_{up}$ is a cyclic 
	$R_{up}$-module generated by  $u+ I_{up}$.
	We denote by  $\vo : \AR_1 \to E$ the volume normalization linear isomorphism
	uniquely specified by the condition that $u+I_{down}$ maps to $1$. 
	In the special case where the field $\fld$ has characteristic~$2$,  one can show
	the {\it Parseval Equations}
	\begin{equation}   \label{eqn!parsevalno1}
		\vo  (x_1) = g_{1,1} \theta_{1,2}  (\vo (x_1))^2 + g_{1,2} \theta_{1,1} (\vo (x_1))^2 +
		g_{2,2} \theta_{1,2} (\vo(x_2))^2
	\end{equation}
	and
	\begin{equation}   \label{eqn!parsevalno2}
		\vo  (x_2) = g_{1,1} \theta_{1,1}   (\vo  (x_1))^2 + g_{1,2}\theta_{1,2}  (\vo   (x_2))^2 +
		g_{2,2} \theta_{1,1} (\vo   (x_2))^2,
	\end{equation}        
	where, for $1 \leq i \leq 2$,  by abuse of notation we denote  $\vo (x_i + I_{down})$ by  $\vo  (x_i)$.
	It is interesting to compare  Equations~(\ref{eqn!parsevalno1}) and (\ref{eqn!parsevalno2})
	with the Parseval Equations obtained in the reference
	\cite[Section~5]{APPS}  and the Differential Identities conjectured in
	\cite[Section~14]{PP} and proven in  \cite[Section~4]{KX}.
\end{example}

\begin{question}     \label{quest!question1}
	The strongest one can hope is that  
	\[
	\Phi(h_1+I_{up}),  \dots  ,  \Phi(h_s+I_{up})
	\] 
	is an $E$-basis of $\Socle(\AR)$.  This happens, by  the results in 
	Section~\ref{sec!gorensteinpolynomialreduction}, 
	if $\fld[x_1, \dots , x_m]/I$  is Gorenstein.  
	It also happens in the torus triangulation 
	Example~\ref{example!torustriangulation}.  However, it fails
	in  Examples~\ref{example!intwovariables} and~\ref{example!inthreevariables}.
	We believe it will be interesting to investigate under which conditions
	on $I$ it holds.
\end{question}

\begin{question}     \label{quest!question2}
	Another interesting question is under which conditions on $I$ 
	the map $\Phi:   N_{up} / I_{up}  \to  \Socle(\AR)$ is injective. 
	By  Corollary~\ref{cor!cohenmacaulayinjctivity} this is the case
	if $\fld[x_1, \dots , x_m]/I$ is  Cohen-Macaulay.  It also  happens in 
	the non-Cohen-Macaulay torus triangulation 
	Example~ \ref{example!torustriangulation},  but it fails
	in Examples~\ref{example!intwovariables} and~\ref{example!inthreevariables}.
\end{question}

\begin{question}     \label{quest!question3}
	Assume that  the dimension of the top degree $\AR_e$
	as $E$-vector space is $1$.  
	We believe it will be worth to investigate under which conditions the 
	$R_{up}$-module $(N_{up}/I_{up})_{(e,-)}$ is cyclic.
	This is the case,  by  the results in 
	Section~\ref{sec!gorensteinpolynomialreduction}, if 
	$\fld[x_1, \dots , x_m]/I$ is Gorenstein,  in the torus triangulation 
	Example~ \ref{example!torustriangulation} and in  
	Example~\ref{example!intwovariables}. However, it fails in an 
	interesting way  in  Example~\ref{example!inthreevariables}.
\end{question}

\subsection{Proof of Proposition~\ref{prop!dimformulaforrupjup}}  \label{subs!dimensionformulaforrupjup} 

\begin{proposition}   \label{prop!gb85i3k9} 
	Assume $R$ is a commutative ring of finite Krull dimension,  
	$I$ is a proper  ideal of $R$,    $s \geq 1$   and   $ f_{(i,1)}, f_{(i,2)}   \in R$,  
	for  $1 \leq i  \leq s$.        We  set,  for  $1 \leq i \leq s$,
	\[
	g_i = f_{(i,1)}  f_{(i,2)}.
	\] 
	We asssume  that   for any sequence  $j_1,   \dots  ,  j_s$, 
	with    $j_t  \in \{1, 2 \}$   for all $t$,   we have 
	\[ 
	\dim R/(I, f_{(1,j_1)},f_{(2,j_2)} ,  \dots  , f_{(s,j_s)} ) = \dim R/I - s.
	\]
	Then, 
	\[
	\dim R/(I, g_1,  \dots ,  g_s )  =  \dim R/I - s.
	\]     
\end{proposition}

\begin{proof}
	By the definition of Krull dimension
	\[
	\dim  R/I = \sup \;  \{  \dim R/p  : p \in  \Spec(R) ,  I \subset p \}.
	\]
	Assume $p \in \Spec R$  with   $( I, g_1,  \dots ,  g_s )  \subset p$.  
	Since   $p$ is prime and  $g_i \in p$   for all $i$,   there exists a sequence
	$\; j_1,   \dots  ,  j_s \;$   
	with    $j_t  \in \{1, 2 \}$   for all $t$,   such that
	\[
	(I, f_{(1,j_1)},f_{(2,j_2)} ,  \dots  , f_{(s,j_s)} )   \subset p.
	\]
	
	By the assumptions
	\[
	\dim R / (I, f_{(1,j_1)},f_{(2,j_2)} ,  \dots  , f_{(s,j_s)}) = \dim R/I -s,
	\]
	hence       $ \dim  R/p    \leq     \dim R/I -s$,  
	which implies that
	\begin{equation}   \label{eqn!34kdialrf}
		\dim R/(I, g_1,  \dots ,  g_s )   \leq  \dim R/I - s.
	\end{equation} 
	
	Conversely,  there exists  $p  \in \Spec (R)$ with   
	\[
	(I, f_{(1,1)},f_{(2,1)} ,  \dots  , f_{(s,1)})  \subset p
	\]
	and        $\dim  R/p =  \dim R/I -s$.    
	Since    
	\[
	(I, g_1,  \dots ,  g_s )  \subset (I, f_{(1,1)},f_{(2,1)} ,  \dots  , f_{(s,1)})  \subset p
	\]
	we get
	\begin{equation}   \label{eqn!94aakkid9}
		\dim R/(I, g_1,  \dots ,  g_s )   \geq  \dim R/I - s            
	\end{equation} 
	Combining   Inequalities~(\ref{eqn!34kdialrf})  and (\ref{eqn!94aakkid9}) the result follows.   
\end{proof}

\begin{remark}   \label{rem!aboutfieldextensions}
	Assume  $\fld$ is a field,   $\fld  \subset F$  is a field extension 
	and $\AR$ is a finitely generated $\fld$-algebra.     Then,
	by    \cite[Tag 00P3]{sproj}
	\[
           	\dim ( \AR  \otimes_\fld   F  ) =  \dim \AR.
	\]
\end{remark} 

\vspace{10pt}

We now give the proof of Proposition~\ref{prop!dimformulaforrupjup}.

\vspace{10pt}

STEP 1.       Using Remark~\ref{rem!aboutfieldextensions}, 
by passing to an infinite field extension of $\fld$ 
(for example the field $\fld(t)$ of rational functions over $\fld$ in one variable)
it is enough to prove the proposition assuming that $\fld$  is infinite.

\vspace{10pt}

STEP 2.     Since   $R_{up}/(I)$   is isomorphic to  
\[
(\fld[x_1, \dots , x_m]  / I) [\theta_{i,j}:   1 \leq i \leq d,  1 \leq j \leq m   ]
\]
we have that
\[
\dim   R_{up}/(I)  =   d + d m.
\]
Hence, by \cite[p. 414, Proposition A.4]{BH}  we have   
\begin{equation}   \label{eqn!439k24h3}
	\dim  R_{up}/I_{up}  \geq    d + d m  - d  = dm.
\end{equation} 

\vspace{10pt}

STEP 3.    Assume  $c_{i,j} \in \fld$  are Zariski general and consider 
the ideal $Q$ of  $R_{up}$, where
\[
Q =   I_{up}  +   (  \theta_{i,j}  - c_{i,j} \theta_{i,1} :   1 \leq i \leq d,  2 \leq j \leq m ).
\]
It is clear that          $R_{up} / Q$    is isomorphic to  
\[
\fld[x_1, \dots , x_m , \theta_{i,1} :  1 \leq i \leq d ]/  ( I,  \theta_{1,1}q_1,  \dots , \theta_{d,1}q_d ),
\]
where,  for  $1 \leq i \leq d$,  we have         $q_i = x_1 +  \sum_{2 \leq j \leq m} c_{i,j} x_j$.           
Using that the $c_{i,j}$  are Zariski general and  \cite[p.~37, Theorem~1.5.17 (c)]{BH},
we get  from Proposition~\ref{prop!gb85i3k9} that
\[
\dim  R_{up}/Q  = d.
\]
Hence,  \cite[p. 414, Proposition A.4]{BH} implies that
\begin{equation}   \label{eqn!irkodlas32}
	\dim  R_{up}/ I_{up}  \leq    d+ (m-1)d = dm.
\end{equation}   
Proposition~\ref{prop!dimformulaforrupjup} follows by combining 
Inequalities~(\ref{eqn!439k24h3}) and (\ref{eqn!irkodlas32}).

\section{KM normalization}        \label{sec!dfnofdegreenormalizationiso}

We continue using the notations of Section~\ref{sec!gericpolynomialreduction}. Motivated by Proposition~\ref{prop!cycliccase},
Example~\ref{example!torustriangulation}, Example~\ref{quest!examplegenerichypersurface} and Remark~\ref{rem!simplicialspheresmotivationgexample},
we give the following definition of volume normalization.

\begin{definition} [KM normalization]    \label{quest!definitioninthecycliccase}
	Assume that  the 
	$R_{up}$-module $(N_{up}/I_{up})_{(e,-)}$ is cyclic,
	and denote by $u$ a generator.
	By  Corollary~\ref{corol!relation758753},  we have that
	$\AR_e$ is a $1$-dimensional vector space over $E$ 
	and  $\Phi(u)$ is an $E$-basis of $\AR_{e}$.
	We define the volume normalization  linear isomorphism $\vo = \vo_u : \AR_e \to E$  
	to be the unique $E$-linear map 
	with the property that $\Phi(u)$ maps to $1$.
\end{definition}

\begin{remark}    \label{rem!uniquenessuptoscalar}
	Assume  that  the   $R_{up}$-module $(N_{up}/I_{up})_{(e,-)}$ is cyclic,
	and denote by $u_1, u_2$ two generators.   Since  $R_{up}$ is bigraded 
	with minimum bidegree $(0,0)$ and $(R_{up})_{(0,0)} = \fld$, it follows that there
	exist $c \in \fld \setminus \{0\}$  such that  $u_2 = c u_1$.  Hence,  
	$\vo_{u_1} = c \vo_{u_2}$.  
\end{remark}

\begin{remark}    \label{rem!aklertiokupa}
	The assumption that   the 
	$R_{up}$-module $(N_{up}/I_{up})_{(e,-)}$ is cyclic is satisfied
	in the important special case that $\fld[x_1, \dots , x_m]/I$ is Gorenstein.
	We study this case in more detail in Section~\ref{sec!gorensteinpolynomialreduction},
	where we relate it with the degenerate Kustin-Miller unprojection which is 
	introduced in   Section~\ref{sec!degeneratekmunprojection}.    
	Moreover, in the even more special case of complete intersections, we give
	explicit computations in Subsection~\ref{subsec!completeintersectioncase}
	and relate it to the theory of multidimensional residues.
\end{remark}

\begin{remark}    \label{rem!simplicialspheres}
	Assume  $\fld[x_1, \dots , x_m]/I$ is the face ring of 
	a simplicial sphere $D$.  In Section~\ref{sec!applicationtosimplicialspheres}
	we  give explicit computations of the volume normalization linear isomorphism 
	and we prove that it is exactly the isomorphism chosen usually in the case of toric varieties, 
and coincides with the one arising from de Rham cohomology. 
\end{remark}

\subsection{The viewpoint of Kustin-Miller unprojection}

Kustin-Miller unprojection is a technique originally invented to construct Gorenstein schemes from simpler ones, see \cite{KM} and \cite{PR}. In 
our setting, it has a new use. We now relate the volume normalization to Kustin-Miller unprojection, therefore justifying the 
name. We refer to p.~\pageref{app!appentixkmunprojection} for a survey of Kustin-Miller unprojection, and concern ourselves only briefly with providing two alternate viewpoints.

We start with a Gorenstein ring $S=\bigslant{\fld[\x]}{I},\ \x=(x_1,\cdots,x_m)$ over $\fld$ of Krull dimension $d$. As a next step, 
we introduce additional variables $\theta_{i,j}$, where $1 \leq j \leq m$  and $1 \leq i \leq d$. We then create new rings, 
by considering the polynomial ring $\fld[\x,\theta_{i,j},z]$ 
and the polynomials \[\theta_i=\sum_{j} \theta_{i,j} x_j.\]
We set
\[R\ = \ \bigslant{\fld[\x,\theta_{i,j},z]}{I+(\theta_i)}\]
and 
\[J\ =\ (z,x_1,\cdots,x_m)\]
The main observation is that we just constructed new Gorenstein rings.
\begin{proposition}
	The rings $R$ and $R/J$ are Gorenstein, and $J$ is of codimension one in $R$
\end{proposition}
 
Using the functor $\mathrm{Hom}_R (-, R)$ to the natural short exact sequence
\[0 \ \longrightarrow\  J \ \longrightarrow\  R\ \longrightarrow\ R/J \ \longrightarrow \ 0\]
we obtain 
\[0 \ \longrightarrow\  R \ \longrightarrow\  \mathrm{Hom}_R (J, R)\ \longrightarrow\ R/J \ \longrightarrow \ 0\]
with the last map corresponding to the Poincar\'e residue map of complex geometry. 

In other words, there is a homomorphism $\varphi:J\longrightarrow R$ that together with the inclusion $\iota:J\longhookrightarrow R$ generates the $R$-module $\mathrm{Hom}_R (J, R)$. Hence, we may find $h$ in $\fld[\x,\theta_{i,j}]$ with $\varphi(z)=h$.

\begin{definition}[KM normalization, the second]
	We set $\vo'(h)=1$.	
\end{definition}

The following is immediate.

\begin{proposition}
	This normalization is well-defined up to an element in $\fld^\ast$, and coincides with the previous definition up to a factor in $\fld^\ast$.
\end{proposition}

\subsection{Once more}\label{ssc:oncemore}

We give a useful alternative view towards this: Consider the polynomial ring
$\fld[\x,\theta_{i,j}]$ and the quotient 
\[T\ =\ \bigslant{\fld[\x,\theta_{i,j}]}{I+(\theta_i)}\]
as well as the quotient
\[U\ =\ \bigslant{\fld[\x,\theta_{i,j}]}{(x_i)}.\]
We have a natural surjection
\[T\ \longrightarrow\ U\]
and hence a map of minimal free resolutions
\[\mathcal{F}_\bullet (T)\ \longrightarrow\ \mathcal{F}_\bullet (U)\]
Consider the top nontrivial entry of the resolution; it is
\[\mathcal{F}_d (T)\ \cong\ \mathcal{F}_d (U)\ \cong\ \fld[\x,\theta_{i,j}]\]
but the map of free resolutions induced
\[\varrho:\mathcal{F}_d (T)\ \longrightarrow\ \mathcal{F}_d (U) \]
is nontrivial: We consider the image $\varrho(1)$, a homogeneous polynomial in the latter, and set
\[\vo''(\varrho(1))=1\].
 From the general theory of Kustin-Miller unprojection (compare the
appendix on p.~\pageref{app!appentixkmunprojection}) we have:

\begin{theorem}
   The maps	$\vo$, $\vo'$ and $\vo''$ coincide (up to a factor in  $\fld^\ast$).
\end{theorem}

\section{Degenerate Kustin-Miller unprojection}  \label{sec!degeneratekmunprojection}  

We refer the reader to the appendix on p.~\pageref{app!appentixkmunprojection} for a survey of Kustin-Miller unprojection. 
In the present section we introduce the notion of degenerate Kustin-Miller unprojection.
Our motivation is that in Section~\ref{sec!gorensteinpolynomialreduction} this notion  
will play a key role in the study of the volume normalization of  commutative Gorenstein algebras.
We remark that an example of degenerate Kustin-Miller unprojection appeared in 
\cite {BP1}.

Assume $R_{small}$ is a local Noetherian ring and $I_{small} \subset J_{small}$ are
two ideals of $R_{small}$ with the property that both quotient  $R_{small} /  I_{small}$
and   $R_{small} /  J_{small}$  are Gorenstein of the same Krull dimension.
We call the pair 
$I_{small} \subset J_{small}$ the {\it  predata for a degenerate Kustin-Miller  unprojection}.

We assume $z$ is a new variable and we set  $R = R_{small}[z]$.  We denote by
$I$ the ideal of $R$ generated by the set $I_{small}$, and by   
$J$ the ideal of $R$ generated by the set $J_{small} \cup \{z\}$.
We fix $h_1,  \dots , h_s \in R_{small}$ that generate the ideal $I_{small}$ of $ R_{small}$, 
clearly they also generate the ideal $I$ of $ R$.

\begin{proposition}  \label{prop!goodunprpair}  We have that  the pair $(J/I, R/I)$ satisfies the 
	Kustin-Miller unprojection assumptions defined in \cite{PR}
	(see also Subsection~\ref{subs!kmunproj433}).  In other words,  
	$R/I$ is  Gorenstein and $J/I$ is a codimension $1$ ideal of $R/I$
	with the quotient being Gorenstein.
\end{proposition}

\begin{proof}
	We denote by $d$ the common Krull dimension
	of $R_{small} /  I_{small}$ and   $R_{small} /  J_{small}$.
	Since $I$ is generated as an ideal of $R$ by the elements $h_i$ of $R_{small}$
	it follows that  $R/I$ is isomorphic to the polynomial ring  $(R_{small} / I_{small}) [z]$.
	This implies that   $R/I$  is  Gorenstein  (since $R_{small} / I_{small}$ is) 
	of Krull dimension $d+1$.  The quotient $(R/I)/(J/I)$ is isomorphic
	to $R_{small} /  J_{small}$ which is Gorenstein of Krull dimension $d$.
\end{proof}

We call the pair $(J/I, R/I)$ 
a {\it degenerate Kustin-Miller  unprojection pair}.
Combining Proposition~\ref{prop!goodunprpair} with  \cite[p.~563]{PR} 
(compare also  Subsection~\ref{subs!kmunproj433})  there 
exists a short exact sequence 
\begin{equation}  \label{eqn!sesforunprno1}
	0\rightarrow \  {R/I}\rightarrow \operatorname{Hom}_{R/I}(J/I, {R/I})
	\rightarrow {R/J}  \rightarrow 0  
\end{equation}
where the first nonzero map sends  $u \in R/I$ to the map $J/I \to {R/I}$ which
is multiplication by $u$ and the second nonzero map corresponds to 
the Poincar\'e residue map.

Recall that if $R$ is a  commutative ring with unit, the subset
$S_R$ of $R$ consisting of all $a \in R$ such that the multiplication by $a$ map
$R \to R$ is injective, is a multiplicatively closed subset of $R$ and the
natural map $R$ to the localization $ (S_R)^{-1}R$, with
$a \mapsto a/1$ is injective. The ring $ (S_R)^{-1}R$ is called the total ring
of fractions of $R$.

We denote by $Q(R/I)$ the total ring of fractions of $R/I$.  Since
the variable $z$ does not appear in the
generating set  $I_{small}$  of $I$, we have that the multiplication by $z$ map
$R/I \to R/I$ is  injective. Consequently,  $z$ is an invertible element of $Q(R/I)$.

\begin{proposition}  \label{prop!mapsendsJinside23}   Assume 
	$ u  \in (  I_{small} :  J_{small}) $.   Then,  the  mutliplication by $u/z$
	map $Q(R/I) \to Q(R/I)$ sends the subset $J/I$ inside  $R/I \subset Q(R/I)$.
	Moreover,   $(u/z) (w+I) = 0$  for all $w \in  J_{small}$.
\end{proposition}

\begin{proof}
	Clearly        $(u/z) (z+I) = u+I \in R/I$.  Assume $w \in J_{small}$,   then
	since  $uw \in I_{small} \subset I$ we have 
	$(u/z) (w+I) = (uw + I) / (z+I) = 0 \in Q(R/I)$.   Since $J$ is generated as
	an ideal of $R$  by  $J_{small} \cup \{z\}$ the proposition follows.
\end{proof}

As a consequence, given  $ u  \in (  I_{small} :  J_{small}) $   there exists 
a well-defined element  $T_u  \in \operatorname{Hom}_{R/I}(J/I, {R/I})$
such that $T_u (z+I) = u$ and   $T_u ( w) = 0$  for all $w \in (J_{small} +I)/I$.

\begin{proposition}  \label{prop!ifkisfaf44}  We  have
	\[
	I \cap  R_{small}  =  I_{small}.
	\]
\end{proposition}

\begin{proof}
	By the definition of $I$ we have  $ I_{small} \subset I$, hence  $ I_{small} \subset (I \cap  R_{small})$.
	For the other inclusion, we assume   $w \in (I \cap  R_{small})$.   Hence, there exist   $m_i \in R$ such that
	\[
	w = \sum_{1 \leq i \leq s}   m_i h_i,
	\]
	with equality in $R$.  Substituting  $0$ for the variable $z$, and using that  $w,h_i \in R_{small}$ 
	we get  that  $w \in I_{small}$.
\end{proof}

\begin{proposition}  \label{prop!mapsdfetrw}   
	Assume   $ \phi  \in \operatorname{Hom}_{R/I}(J/I, {R/I})$ has the properties
	$\phi ( w) = 0$  for all $w \in (J_{small} +I)/I$ and that there exists 
	$u \in R_{small}$ such that   $\phi ( z+I) = u+I$.  Then 
	$ u  \in (  I_{small} :  J_{small}) $  and   $\phi = T_u$.
\end{proposition}

\begin{proof}
	Assume  $w \in J_{small}$.  We have 
	\[
	0 = (z+I) \phi (w+I)  = \phi ( (z+I) (w+I) )  =  (w+I)  \phi  (z+I)   =  wu+ I.
	\]
	Hence $wu \in I$.   Since $wu \in  R_{small}$, it follows that 
	$wu \in  (I \cap  R_{small})$.  Using Proposition~\ref{prop!ifkisfaf44}
	$wu \in I_{small}$, which implies that
	$ u  \in (  I_{small} :  J_{small}) $.  Since $J$ is generated as
	an ideal of $R$  by  $J_{small} \cup \{z\}$ the equality $\phi = T_u$   follows. 
\end{proof}

We will now prove that the second assumption in the previous proposition 
implies the first.

\begin{proposition}  \label{prop!mapsdfetrewrw}   Assume 
	$ \phi  \in \operatorname{Hom}_{R/I}(J/I, {R/I})$ has the 
	property that there exists 
	$u \in R_{small}$ such that   $\phi ( z+I) = u+I$.  Then 
	$\phi ( w) = 0$  for all $w \in (J_{small} +I)/I$. In addition,  
	$ u  \in (  I_{small} :  J_{small}) $   and   $\phi = T_u$.
\end{proposition}

\begin{proof}
	Using Proposition~\ref{prop!mapsdfetrw} it is enough to prove that
	$\phi ( w) = 0$  for all $w \in (J_{small} +I)/I$.

	Assume $w \in J_{small}$. We fix $r \in R$ with  $\phi(w+I) = r + I$.    We have  
	\[
	zr+I =   (z+I) \phi (w+I) =   \phi ( (z+I) (w+I) )  =  (w+I)  \phi  (z+I)   =  wu+ I.
	\]
	Hence,  $wu - zr \in I$, which implies that there exist  $m_i \in R$ such that
	\[
	wu  - zr =  \sum_{1 \leq i \leq s}   m_i h_i
	\]
	with equality in $R$.   We write,  for $1 \leq i \leq s$, 
	\[
	m_i =  \sum_{1 \leq i \leq s}   m_{1,i} z + m_{2,i},
	\]
	with  $m_{1,i} \in R$ and $m_{2,i} \in R_{small}$. Hence,
	\[
	wu  - zr =  z \sum_{1 \leq i \leq s}   m_{1,i} h_i +  \sum_{1 \leq i \leq s}   m_{2,i} h_i.
	\]
	Since $w,u, h_i, m_{2,i}  \in   R_{small}$,   looking at the coefficients of $z$
	it follows that 
	\[
	r =    -\sum_{1 \leq i \leq s}   m_{1,i} h_i,  
	\]
	with equality in $R$.  Hence $r \in I$, which implies that  $\phi ( w + I) = 0$.
\end{proof}

\begin{proposition}  \label{prop!msdfolxdgs}   Assume 
	$ \phi  \in \operatorname{Hom}_{R/I}(J/I, {R/I})$.   Then there exist
	$c \in R$ and an element $u \in (  I_{small} :  J_{small})$ such that 
	\[
	\phi = T_u  +  c i,
	\]
	where  $i :  J/I  \to {R/I}$  denotes the natural inclusion.
\end{proposition}

\begin{proof} 
	We fix $r \in R$ such that  $\phi  (z+I) = r+ I $.  We write
	\[
	r =   c z + u,  
	\]
	with  $c \in R$ and $u \in R_{small}$ and set 
	$\phi_1 = \phi -  c  i$.  We have that 
	$\phi_1 \in \operatorname{Hom}_{R/I}(J/I, {R/I})$ and 
	\[
	\phi_1  (z+I)  = \phi  (z+I) - c (z+I)  = u + I.
	\]
	By Proposition~\ref{prop!mapsdfetrewrw},  
	$ u  \in (  I_{small} :  J_{small}) $   and   $\phi_1 = T_u$.
\end{proof}

\begin{proposition}  \label{prop!tesldjkasd623}   Assume  $u, v \in 
	(  I_{small} :  J_{small})$.  Then $T_u = T_v$  if and only if 
	$u-v \in I_{small}$.
\end{proposition}

\begin{proof}   
	Assume first that    $u-v \in I_{small}$.  Since $I_{small} \subset I$
	we get  $u + I = v  + I$.  Since  $J$ is the ideal of $R$ generated by the subset 
	$J_{small} \cup \{z\}$ and both $T_u $ and  $T_v $ are zero on  $(J_{small} +I)/I$,
	to prove $T_u = T_v$ it is enough to prove that if $T_u(z+I) = T_v(z+I) $.  This holds,
	since  $T_u(z+I) = u+I$  and  $T_v(z+I) = v +I $.
	
	Conversely, we assume that  $T_u = T_v$.  Evaluating at $z+I$ we get 
	\[
	u + I  =  T_u(z+I) =T_v(z+I) = v +I,
	\]
	which implies that  $u-v \in I$.  By assumption $u-v \in R_{small}$.  Proposition~\ref{prop!ifkisfaf44}
	now implies that $u-v \in I_{small}$. 
\end{proof}

We consider the $R_{small}$-module  
\[
L =  (  I_{small} :  J_{small}) /  I_{small}.
\]

\begin{theorem}  \label{prop!rnaedgssd} 
	The $R_{small}$-module   $L$ is cyclic.  Assume 
	$u \in (  I_{small} :  J_{small})$ has the property 
	that $u +  I_{small}$  generate  $L$.   Then 
	$T_u$ together with inclusion map  $i : J/I \to  R/I$
	generate the $R/I$-module $\operatorname{Hom}_{R/I}(J/I, {R/I})$.
\end{theorem}

\begin{proof}
	Using the short exact sequence $(\ref{eqn!sesforunprno1})$, there exists 
	$\phi \in \operatorname{Hom}_{R/I}(J/I, {R/I})$ that together 
	with the inclusion generate  $\operatorname{Hom}_{R/I}(J/I, {R/I})$
	as $R/I$-module.  Using Proposition~\ref{prop!msdfolxdgs}, there exist 
	$c \in R$ and $v \in (  I_{small} :  J_{small})$ such that 
	\[
	\phi = T_v  +  c i.
	\]
	As a consequence, $T_v$  together with the inclusion generate  $\operatorname{Hom}_{R/I}(J/I, {R/I})$
	as $R/I$-module. 
	
	We claim that $v+  I_{small}$ generate the  $R_{small}$-module  $L$.  Indeed,
	assume $w \in  (  I_{small} :  J_{small})$.   Then $T_w \in \operatorname{Hom}_{R/I}(J/I, {R/I})$,
	hence there exist  $c_1,c_2 \in R$ such that
	\[
	T_w = c_1 T_v +  c_2  i. 
	\]
	Consequently
	\[
	w + I = T_w  (z+I) = c_1 T_v  (z+I)  +  c_2  i(z+I) = (c_1v + c_2 z) + I.
	\]
	Hence, there exist $m_i \in R$ such that 
	\[
	w -  (c_1v + c_2 z) =   \sum_{1 \leq i \leq s}   m_{i} h_i,
	\]
	with equality in $R$.  Substituting  $0$ for the variable $z$,  using that  $w,h_i, u \in R_{small}$ 
	we get that $w -c_3 v  \in I$, where  $c_3  \in  R_{small}$  is the result of the 
	substitution of $0$ for $z$ in $c_1$.
	Since $w,c_3, v \in R_{small}$,  Proposition~\ref{prop!ifkisfaf44}
	implies that $w-c_3v \in I_{small}$.   Hence,   $w+ I_{small} =  c_3v + I_{small}$, therefore
	$L$ is cyclic with generator $v+  I_{small}$.
	
	Conversely, assume that  $u \in (  I_{small} :  J_{small})$ has the property 
	that $u +  I_{small}$  generate  $L$. This implies that there exists $c_4 \in R_{small}$ such that
	\[
	v+ I_{small} = c_4  (u+I_{small}),  
	\]
	hence 
	\[
	T_v = c_4 T_u.
	\]
	Since $T_v$  together with the inclusion generate  $\operatorname{Hom}_{R/I}(J/I, {R/I})$
	we get that also  $T_u$  together with the inclusion generate  $\operatorname{Hom}_{R/I}(J/I, {R/I})$.
\end{proof}

\begin{remark}  \label{rem!converseofProposition} 
	Assume $u \in (  I_{small} :  J_{small})$ has the property 
	that $T_u$ together with inclusion map  $i : J/I \to  R/I$
	generate the $R/I$-module $\operatorname{Hom}_{R/I}(J/I, {R/I})$.
	The arguments in the proof of  Theorem~\ref{prop!rnaedgssd} 
	show  that   $u +  I_{small}$  generates  $L$ as 
	$R_{small}$-module.
\end{remark}

\begin{remark}  Assume  $u \in (  I_{small} :  J_{small})$ has the property 
	that $u +  I_{small}$  generate  $L$.  This implies that 
	\[
	(  I_{small} :  J_{small}) =  I_{small} +  (u),
	\]
	where $(u)$ denotes the ideal of $R_{small}$ generated by $u$. 
	Then 
	\[
	(  I_{small} :  (u))  =    (  I_{small} :  (I_{small} + (u) ))   = 
	(  I_{small} :   (  I_{small} :  J_{small}))   = J_{small},
	\]
	where  for the last equality we used  \cite[p. 115, Proposition~5.2.3~d)]{Mi}.
\end{remark}

\phantom{===}

\section{The Gorenstein case of volume normalization}  \label{sec!gorensteinpolynomialreduction}

In the present section we discuss
the volume normalization theory of   $\fld[x_1, \dots , x_m]/I$ 
under the additional assumption that the quotient ring is Gorenstein.
We base our treatment on the notion of degenerate Kustin-Miller unprojection
introduced in Section~\ref{sec!degeneratekmunprojection}.

Assume $m \geq 1$ and  $\fld$ is a field.  We consider the polynomial 
ring  $\fld[x_1, \dots , x_m]$, where the degree of the variable $x_i$ is equal to $1$,
for all $1 \leq i \leq m$.  Assume   $I \subset \fld[x_1, \dots , x_m]$ 
is a homogeneous ideal with   $\fld[x_1, \dots , x_m]/I$ Gorenstein.
We denote by $d$ the   Krull dimension of the quotient ring  $\fld[x_1, \dots , x_m]/I$ 
and  assume  $d \geq 1$.

We will use the notations $R_{up}, I_{up},R_{down}, I_{down},\AR,e, S, E, \Phi,$
defined in  Section~\ref{sec!gericpolynomialreduction}.
Using the assumption that the  quotient ring  $\fld[x_1, \dots , x_m]/I$  is Gorenstein,  we get by  
Proposition~\ref{prop!dimformulaforrupjup} that    $R_{up}/ I_{up}$ is Gorenstein, since it is the quotient of 
$R_{up}/ (I)$ by a homogeneous regular sequence.    Consequently, 
the pair   $I_{up} \subset ( x_1, \dots , x_m)$ of ideals of $R_{up}$ 
is the predata for a degenerate  Kustin-Miller  unprojection 
(see Section~\ref{sec!degeneratekmunprojection}).

By  Theorem~\ref{prop!rnaedgssd},  the $R_{up}$-module 
$ (  I_{up} :  ( x_1, \dots , x_m)) /  I_{up}$  is cyclic, and we fix an element
$h \in (  I_{up} :  ( x_1, \dots , x_m))$ such that  $h + I_{up}$ generates
$(  I_{up} :  ( x_1, \dots , x_m)) /  I_{up}$.   We get the following
equality of ideals of $ R_{up}$
\begin{equation} \label{eqn!jkorafi}
	(  I_{up} :  ( x_1, \dots , x_m) )  =    I_{up} + (h).
\end{equation}

\phantom{===}

\begin{theorem}    \label{thm!ulocisnonzero}
	The element  $\Phi(h + I_{up})$  is a nonzero element of  $\Socle(\AR)$.
	Hence,  it is a basis of $\Socle(\AR)$ considered 
	as $E$-vector space.  
\end{theorem}

\begin{proof}      
	Since $\AR$ is a graded Artinian Gorenstein  $E$-algebra, 
	$\Socle(\AR)$ is a $1$-dimensional $E$-vector space and 
	$\Socle(\AR) = \AR_e$.  The result follows by Corollary~\ref{corol!relation758753}.
\end{proof}

Using the theorem,  there exists a unique  volume normalization linear isomorphism $\AR_e \to E$
such that it sends  $\Phi(h + I_{up})$ to $1$.

\phantom{===}

\subsection{The complete intersection case of volume normalization and relation with residues}  \label{subsec!completeintersectioncase}

In this subsection we discuss the volume normalization theory of the algebra   $\fld[x_1, \dots , x_m]/I$ 
under the additional assumption that $I$ is a homogeneous
complete intersection ideal.  In addition, we relate it to the theory of multidimensional residues.

Assume $m \geq 1$ and  $\fld$ is a field.   We consider the polynomial 
ring  $\fld[x_1, \dots , x_m]$, where the degree of the variable $x_i$ is equal to $1$,
for all $1 \leq i \leq m$.   Assume   $d \geq 1$ and 
\[
I =  (H_1, \dots  , H_{m-d} )  \subset \fld[x_1, \dots , x_m]
\]  
is a homogeneous ideal with  the property  that 
$H_1, \dots  , H_{m-d}$  is a homogeneous regular sequence  of $\fld[x_1, \dots , x_m]$.
As a consequence,  the  Krull dimension of the quotient ring  $\fld[x_1, \dots , x_m]/I$ is $d$.

We will use the notations $R_{up}, I_{up},\AR,e,  E, \Phi,$
defined in  Section~\ref{sec!gericpolynomialreduction}.

Using the assumption that the  $H_1, \dots  , H_{m-d}$  is a homogeneous regular sequence,
we get that  
\[ 
H_1, \dots  , H_{m-d}, f_1,  \dots , f_d 
\]
is a homogeneous regular sequence  of  $R_{up}$. In particular,
$R_{up}/ I_{up}$ is Gorenstein.

We fix an  $m \times m$ matrix $M$ such that 
\[
(H_1, \dots  , H_{m-d}, f_1,  \dots , f_d )^t  =  M  (x_1, \dots  , x_m)^t,
\]
where $N^t$  denotes the transpose of the matrix $N$. We denote by
$\Delta$ the determinant of the matrix $M$.
Combining the computations in \cite[Section~4]{P} 
(see also Subsection~\ref{subs!citoci})  with the discussion in
Section~\ref{sec!gorensteinpolynomialreduction}  we get that 
\[
(  I_{up} :  ( x_1, \dots , x_m) )  =    I_{up} + (\Delta).
\]
As a consequence, we have the following proposition.

\begin{proposition}  \label{prop!degnorminthecicase}   
	There exists a  volume normalization linear isomorphism $\AR_{e} \to E$
	uniquely specified by the property that $\Phi (\Delta)$ maps to $1$.
\end{proposition}

Comparing  Proposition~\ref{prop!degnorminthecicase}   with
\cite[p.~42, Equation~(1.50)]{CD} we get a connection, in the complete intersection setting,
between the volume normalization theory and the theory of multidimensional residues.

In particular, we can generalize the projective residue map
$\; \res^{\mathbb{P}^n}_{<F_0, \dots , F_n>} \;$  defined in \cite[p.~40]{CD}
as follows.

Assume $m \geq 0$ and  $\fld$ is a field.   We consider the polynomial 
ring  $R= \fld[x_0, \dots , x_m]$, where the degree of the variable $x_i$ is equal to $1$,
for all $0 \leq i \leq m$.   Assume   $ I \subset R$ is a
homogeneous ideal such that the quotient  $R/I $ is Artinian Gorenstein.
We set $J = (x_0, \dots , x_m)$.
Then, the pair  $I  \subset  J$ is a predata
for a degenerate Kustin-Miller unprojection 
(see Section~\ref{sec!degeneratekmunprojection}).
By  Theorem~\ref{prop!rnaedgssd},  the $R$-module 
$ N  =  (  I :  J) /  I$  is cyclic, and we fix an element
$h \in (I :  J)$ such that $h+I$ generates $N$ as an $R$-module.

\begin{definition}    \label{dfn!rprojectiveresidue}
	Assume $Q \in R$ is a homogeneous element.
	If there exists $\lambda \in \fld$ with  $Q = \lambda h$
	(equality in $R/I$), we set 
	$\; \res^{\mathbb{P}^n}_{I,h} (Q) = \lambda$, otherwise
	we set    $\; \res^{\mathbb{P}^n}_{I,h} (Q) = 0$.
\end{definition}

\begin{remark}    \label{rem!dependenceuptosign}
	Using  Remark~\ref{rem!uniquenessuptoscalar} we have the following.
	Assume $h_1, h_2 \in (  I :  J )$  have the property that, for all $1 \leq i \leq 2$, the element
	$h_i + I$ generates $N$ as an $R$-module.   Then there exists unique $c \in \fld \setminus \{ 0  \}$
	such that
	\[
	\res^{\mathbb{P}^n}_{I,h_2} = c   \; \res^{\mathbb{P}^n}_{I,h_1} 
	\]
	In particular, if the field $\fld$ has two elements then  Definition~\ref{dfn!rprojectiveresidue} 
	does not depend on the choice of $h$.
\end{remark}

\phantom{===}

\section{Volume normalization in the case of simplicial spheres}    \label{sec!applicationtosimplicialspheres}

In the present section we compute the volume normalization 
of the face ring of a simplicial sphere.

Assume $\fld$ is a field and $D$ is a  simplicial sphere
with vertex set  $\{1, 2, \dots, m\}$  and dimension
equal to  $d-1 \geq 1$. 
We  denote by $I_D \subset \fld[x_1, \dots  ,x_m]$   the Stanley-Reisner ideal of $D$
and by $\fld[D]= \fld[x_1, \dots  ,x_m]/I_D$ the face ring of $D$ over $\fld$.
We set  
\[
R_{up} =\fld[x_1,\dots, x_m, \theta_{i,j}: 1\leq i\leq d, 1\leq j\leq m],
\] 
and,  for $1 \leq i \leq d$, we set 
\[
f_i = \sum_{j=1}^m \theta_{i,j} x_j \in R_{up}.
\]
We denote by  $I_{up}$ the ideal of $R_{up}$ generated by $I_D \cup  \{f_1,\dots, f_d \}$,
and we set
\[
L =  (  I_{up} :  (x_1,  \dots, x_m)) /  I_{up}.
\]

Since  $D$ is a simplicial sphere, by  \cite[Section~5]{BH}  $\fld[D]$ is Gorenstein
of Krull dimension $d$. As a consequence,  the pair 
$I_{up} \subset (x_1,  \dots, x_m)$ is a  predata for a degenerate Kustin-Miller  unprojection
(see Section~\ref{sec!degeneratekmunprojection}).

\begin{proposition}   \label{propos!liscyclic}
	The $R_{up}$-module $L$ is cyclic. 
\end{proposition}

\begin{proof}
	It is an immediate application of  Theorem~\ref{prop!rnaedgssd}.
\end{proof}

We will now describe a generator of the cyclic module $L$.
We fix an ordered  facet   $\sigma=(b_1, \dots , b_d)$  of $D$.  This just means 
that the set  $\{b_1,  \dots ,b_d\}$  is a facet of $D$.
We set  $x_{\sigma}=  \prod_{1 \leq i \leq d} x_{b_i} \in R_{up}$. Moreover,
we denote by $[\sigma]$ the determinant of the $d \times d$ matrix with 
$(i,j)$-entry equal to $\theta_{i,b_j}$.  Finally, we set 
\begin{equation}  \label{eqn!definitionofu}
	h =  [\sigma] x_{\sigma} \in  R_{up}.
\end{equation}
The proof of the following theorem will be given in Subsection~\ref{subs!pfoftheoremeeree}.

\begin{theorem}   \label{theorem!eeree}
	The element $h + I_{up}$ is a generator of the cyclic $R_{up}$-module $L$.
\end{theorem}

\begin{remark} \label{rem!simplicialspheresmotivationgexample}
	We use the notations $E=\fld(\theta_{i,j}), R_{down}= E[x_1, \dots , x_m] , 
	I_{down} = (I_D) + (f_1, \dots , f_d) \subset  R_{down} , \AR=R_{down}/I_{down} \; $  
	defined in Section~\ref{sec!gericpolynomialreduction}.
	We remark that (up to sign)  the volume normalization 
	linear isomorphism  $\AR_{d} \to E$   
	defined in  \cite[p.~6, Equation (1)]{PP} (and denoted there by the notation
	$\Psi_{e}$)
	is  uniquely specified by the condition that the element   $\; [\sigma] x_{\sigma}  + I_{down} \;$  maps to  $1$.
\end{remark}

\subsection {Proof of  Theorem~\ref{theorem!eeree}}   \label{subs!pfoftheoremeeree}

We assume $z$ is a new variable and we set  $R=R_{up}[z]$, hence
\[   
R= \fld[z, x_1,\dots, x_m,  \theta_{i,j}: 1\leq i\leq d, 1\leq j\leq m].
\]
We consider the ideals  $I= (I_D)+ (f_1,\dots, f_d)$, $J= (x_1,\dots, x_m, z)$  
and  $P= (I_D)$  of $R$.   Recall, that for any $w \in (I_{up} : (x_1,  \dots, x_m))$
we defined in Section~\ref{sec!degeneratekmunprojection}
an element $T_w \in \operatorname{Hom}_{R/I}(J/I, R/I)$. 
We will prove in Proposition~\ref{clm75!qual}  that  $h \in   (I_{up} : (x_1,  \dots, x_m))$.
Using that and Remark~\ref{rem!converseofProposition}, to show Theorem~\ref{theorem!eeree}
it is enough to prove Proposition~\ref{prop!iogsdg} below,  which states that
the map $T_{h}$ together with the inclusion 
$i$ generate $\operatorname{Hom}_{R/I}(J/I, R/I)$ as $R/I$-module.

We set 
\[
c= \operatorname{codim} \, {I_D}= \dim \fld[x_1, \dots  ,x_m] - \dim (\fld[x_1, \dots  ,x_m]/I_D)= m-d.
\]
We use the following bidegree for  $R$. For $a,b \in \mathbb{Z}$, we say that an element  $w \in R$
has bidegree $(a,b)$  if it is homogeneous of  degree $a$ with respect to the variables $x_i$ and $z$
and  is homogeneous of degree $b$ with respect to the variables $\theta_{i,j}$.  We denote by $R_{a,b}$ the subset
of $R$ consisting of bihomogeneous elements of bidegree $(a,b)$, of course $R_{a,b}=0$ if
$a <0$ or $b <0$. For $a,b \in \mathbb{Z}$ we denote  by $R(a,b)$ the bigraded
$R$-module with  $R(a,b)_{k,l}= R_{a+k, b+l}$ for all  $k,l \in \mathbb{Z}$.

\vspace{1.8mm}

We denote by

\[C^{1}: \text{ minimal graded free resolution of $R/P$ as $R$-module.}\]

\[C^{2}: \text{ minimal graded free resolution of  $R/(f_1,\dots, f_d)$ as $R$-module.}\]

\[C^{3}: \text{ minimal graded free resolution of $R/I$ as $R$-module.}\]

\[C^{4}: \text{ minimal graded free resolution of $R/J$ as $R$-module.}\]

\vspace{1.0mm}

\begin{proposition}
	$C^{1}$ is of the form
	\[0\rightarrow F_{c} = R(-m,0)\rightarrow F_{c-1}\rightarrow \dots \rightarrow F_{1}\rightarrow F_{0}=R\]
\end{proposition}

\begin{proof} 
	By  \cite[Lemma 5.6.4]{BH} 
	we have that the socle degree of $R/P$  is equal to $ d = \dim D +1$ and 
	$d = m-c$.  The result follows by the discussion in  \cite[p.~79]{Mi}.
\end{proof}

\vspace{1.0mm}

\begin{proposition}
	$C^{2}$ is the Koszul complex on $(f_1,  \dots , f_d)$.  Hence, $C^{2}$ is
	\[0\rightarrow  \tilde{F_{d}}=R(-d,-d)\rightarrow \tilde{F}_{d-1}\rightarrow \dots \rightarrow \tilde{F_{1}}\rightarrow \tilde{F_{0}}=R\]
\end{proposition}

\vspace{1.0mm}

\begin{proposition}\label{clm!tensprod}
	$C^{3}$ is the tensor product of $C^{1}$  and  $C^{2}$.
\end{proposition}

\begin{proof} 
	It follows by combining \cite[p.~ 8]{DH}  with \cite[Exercise 1.1.12]{BH} and  \cite[Theorem 9.4.7]{BH}.
\end{proof}

\vspace{1.0mm}

Hence, by Proposition~\ref{clm!tensprod} it follows that

\vspace{1.0mm}

\begin{proposition}
	$C^{3}$ is
	\[0\rightarrow  \hat{F_{m}}=R(-m-d,-d)\rightarrow \hat{F}_{m-1}\rightarrow \dots \rightarrow \hat{F_{1}}\rightarrow \hat{F_{0}}=R\]
\end{proposition}

\begin{proposition}
	$C^{4}$ is the Koszul complex on $(x_1,\dots, x_m, z)$.   Hence, $C^{4}$ is
	\[0\rightarrow  \bar{F}_{m+1}=R(-m-1, 0)\rightarrow \bar{F}_{m}=(R(-m, 0))^{m+1}\rightarrow \dots \rightarrow \bar{F_{1}}\rightarrow \bar{F_{0}}=R\]
\end{proposition}

\begin{proposition}\label{clm56!who}
	Since 
	\[I  \subset J  \subset    R\]
	there is a natural projection map
	\[R/I \rightarrow R/ J. \]
	It induces a chain map $C^{3}\rightarrow C^{4}$.
	Looking at 
	\begin{tikzcd}
		\hat{F_{m}}\simeq R \arrow[r, "h"]
		&  \bar{F}_{m}\simeq R^{m+1}
	\end{tikzcd}
	each entry of the matrix of  $h$ has bidegree $(d,d)$.
\end{proposition}

\begin{proposition}\label{clm6!equal}
	If $\tau$ is an ordered facet of $D$, then there exists $\epsilon\in \{-1,1\}$ such that
	\[
	h - \epsilon[\tau]x_{\tau}  
	\]
	is an element of the ideal $(f_1, \dots , f_d)$  of $R$.
\end{proposition}

\begin{proof}
	See   \cite[Corollary 4.5]{PP}.
\end{proof}

\begin{proposition}\label{clm7!qual} 
	The element  $h+I$ of  $R/I$  is nonzero.
\end{proposition}

\begin{proof} We use the notations  $R_{down}= \fld (\theta_{i,j})[x_1, \dots , x_m] , 
	I_{down} = (I_D) + (f_1, \dots , f_d) \subset  R_{down}$  
	defined in Section~\ref{sec!gericpolynomialreduction}.
	By \cite[p.~6]{PP},    $h+ I_{down}$ is a nonzero element of   $R_{down}/I_{down}$.   This implies that the 
	element  $h+I$ of  $R/I$  is nonzero. 
\end{proof}

\begin{proposition}\label{clm75!qual} 
	We have   $h \in (I_{up} :(x_1,  \dots, x_m))$.
\end{proposition}

\begin{proof}
	Assume $1 \leq j \leq m$.  We choose a facet $\tau$ of $D$ such that $j$ is not an element of $\tau$.
	Using Proposition~\ref{clm6!equal}, there exists $\epsilon\in \{-1,1\}$ such that
	\[
	h - \epsilon[\tau]x_{\tau}    \in  I_{up}.
	\]
	Moreover,  since  $j$ is not an element of the facet $\tau$, we have  that
	$\tau \cup \{j \}$ is not a face of $D$, hence
	\[
	x_{\tau}x_{j}  \in I_D \subset I_{up}.
	\]
	Hence,  working on the quotient ring $R/I_{up}$ we have
	\[
	h x_j +  I_{up}  =   \epsilon [\tau]x_{\tau}x_{j} + I_{up} 
	=  0 + I_{up}.
	\]
	This finishes the proof of the proposition.
\end{proof}

\begin{proposition}   \label{prop!iogsdg}
	The map $T_{h}$ together with the inclusion 
	$i$ generate the $R/I$-module $\operatorname{Hom}_{R/I}(J/I, R/I)$.
\end{proposition}

\begin{proof}
	We denote by $<i>$ the $R/I$-submodule of $\operatorname{Hom}_{R/I}(J/I, R/I)$
	generated by the inclusion $i$.
	Proposition~\ref{clm56!who} together with  the short exact sequence~(\ref{eqn!sesforunprno1})
	and the general  theory of Kustin-Miller unprojection  (see  \cite{P,PR} and  Subsection~\ref{sec!kmorig})   
	implies  that 
	\[    
	\operatorname{Hom}_{R/I}(J/I, R/I)/<i>
	\]
	is a cyclic module generated  by the class of an element $\phi$ such that  
	$\phi (z)$  has  bidegree $(d,d)$.  
	By  Proposition~\ref{clm75!qual}  $h \in   (I_{up} : (x_1,  \dots, x_m))$,
	$h$ clearly  has bidegree $(d,d)$ and by   Proposition~\ref{clm7!qual} $h$ is nonzero in $R/I$. 
	This finishes the proof of Proposition~\ref{prop!iogsdg}.
\end{proof}

\section{The relative case and locality}

The volume normalization theory we have developed so far can be used 
under the assumption that we work with Gorenstein algebras. It applies, 
for example, for the class  of  face rings of simplicial spheres.  
The relative volume normalization setting, which we will now discuss,
extends the theory.  It is useful to handle 
situations where  the socle is not one-dimensional and one may want 
to normalize with respect to a specific choice of fundamental class (such 
is the case when considering face rings of triangulated
manifolds \cite{AHL, APP}).   Moreover, it allows  us to understand the 
normalization in cases where it otherwise becomes rather challenging to 
access it directly, as in the case of lattice polytopes.

Let us try to establish that normalization is, in some way, local. We already saw a glimpse of this: In the case of simplicial spheres $D$,
 the volume map evaluated at $\mbf{x}_F$, where $F$ is a facet, only depended on the linear system of parameters
 at that facet; more precisely, on the indeterminates corresponding to that facet. But the why of this is a bit unclear from our calculation. We aim to explain that fact:

There are several ways to attempt to encode this: For instance, we could consider simplicial homeomorphisms $D'\rightarrow D$ 
that leave $F$ invariant, and study the resulting map of rings. Alas, the class of objects we then understand is not of the desired 
generality. This section's main goal is to observe that sometimes, in order for practical calculations to succeed, it is useful to consider relative settings.

So we adopt a different perspective: consider a standard graded Gorenstein ring $T$ as a module over a polynomial 
ring $R=\fld[\mbf{x}]$. We assume once more that $T$ is of positive Krull dimension, and that the fundamental class
 lives in degree $d$. We finally consider a squarefree monomial $\mbf{x}_F$ of degree $d$ in $R$, and set out to compute it's volume map with respect to a normalization. 

We call $\mbf{x}_F$ \emph{isolated} in $T$ if the ideal $I_F$ generated by $x_F$ in $T$ is isomorphic to $\fld[\mbf{x}_{i,i\in F}]$.

\begin{example}
Consider a simplicial complex $X$, and let $F$ denote a facet of $X$. Then $\mbf{x}_F$ is isolated in the face ring of $X$.
\end{example}

Naturally, the Krull dimension of $T$ is $d$. 

Moreover, $T/I_F$ is Cohen-Macaulay, and $I_F$ is Gorenstein. 

Hence, the question arises:

What is the normalization, the value of the volume map, at $\mbf{x}_F$?

We can ask this question twice:

What is the value in $I_F$ (which, naturally, is Gorenstein itself)?

And what is the value in $T$?

Of course, we expect the answers to be the same. Let us do the math.

\subsection{Relative KM Normalization in $I_F$}\label{sec:relI}

We adopt the viewpoint of Section \ref{ssc:oncemore}. Consider $I_F=\langle \mbf{x}_F \rangle$ in $T$. 

Let $R_{up}=R[\theta_{i,j}]$. Consider the minimal free resolution $F_\bullet$ of the module $I_F/\Theta I_F$. 

\[ 0\ \longrightarrow\ F_d\ \longrightarrow\ F_{d-1}\ \longrightarrow\ \cdots\ \longrightarrow F_{0}\ \longrightarrow\ \bigslant{I_F}{\Theta I_F}\ \longrightarrow\ 0 \]

We  also consider  the minimal resolution $G_\bullet$ of ${R_{up}}/{\mbf{x}R_{up}}$. Since we have a natural map 
\[I_F=F_0\ \longrightarrow\ R_{up}=G_0\]
induced by the inclusion of ${I_F}$ into $S'$,
and in particular a map of the one-dimensional modules $\Phi:F_d\rightarrow G_d$, both of which are isomorphic to $R_{up}$. 

\begin{definition}
The relative KM normalization sets $\vo(\Phi(1))=1$.
\end{definition}

We have that $\Phi(1)=\Delta(\theta_{i,j})_F \mbf{x}_F$, and hence immediately obtain the KM normalization 
\[\vo(\Delta(\theta_{i,j})_F \mbf{x}_F)=1.\]

\subsection{Normalization in $T$}\label{sec:NorminI}

We want to consider now $T$ as an $R_{up}$-module, and compute the minimal free resolution $H_\bullet$ of $T/\Theta T$. Note: we have a commutative triangle 
\[\begin{tikzcd}[column sep=5em]
	\bigslant{I_F}{\Theta I_F} \arrow[hook]{r}{} \arrow{dr}{} & \bigslant{T}{\Theta T} \arrow{d}{} \\
	 & \bigslant{R_{up}}{\mbf{x}R_{up}}
\end{tikzcd}
\]
which of course induces a commutative triangle 
\[\begin{tikzcd}[column sep=5em]
	F_d \arrow{r}{} \arrow{dr}{} & H_d \arrow{d}{} \\
	& G_d
\end{tikzcd}
\]
of one dimensional $R_{up}$-modules. The top map is an isomorphism, hence the normalization coincides. 

\begin{example}
If $D$ is a simplicial sphere, and $F$ is a facet, it follows that the KM normalization satisfies
\[\vo(\Delta(\theta_{i,j})_F \mbf{x}_F)=1.\]
\end{example}

\subsection{Locality}

Let us reexamine what happened in the last section, and for this purpose, introduce the notion of 
a \emph{dualizable Cohen-Macaulay module}: Such modules are characterized by being Cohen-Macaulay, and the socle 
of their Artinian reduction is of dimension one.

One such example is considering a Gorenstein ring $T$ as above, and considering an ideal $K$ in it 
such that $T/K$ is Cohen-Macaulay. Such a module is dualizable Cohen-Macaulay, and, by defining the relative 
normalization as in Section~\ref{sec:relI}, with $I_F$ replacing $K$. Section~\ref{sec:NorminI} then shows that this 
normalization only depends on $K$, not on $T$. In particular, if two rings contain isomorphic dualizable Cohen-Macaulay ideals, then 
the KM normalization in this ideal only depends on $K$, and not on $T$ or $T'$. Another way of saying it: If $m$ is an element 
in the socle of $\mathcal{K}$, the Artinian reduction of $K$, then $\vo (m)$ is independent of whether we are considering $T$ or $T'$.

\section{Applications}

Let us note the two main applications of this relative perspective.

\subsection{Simplicial spheres to manifolds}

One issue we had in the absolute perspective was to explain the KM normalization in the case of objects 
more general than manifolds, see \ref{example!torustriangulation}; the relative normalization helps here. We can 
use it to explain the normalization for general (closed, compact orientable) manifolds, pseudomanifolds and more generally face rings of cycles \cite{APP}.

Recall: Given a triangulated $(d-1)$-dimensional cycle $\mu$ over $\fld$, we can consider the face ring $T$ of its underlying simplicial complex. 
Consider any Artinian reduction $\AR$ of $T$. Note that $\mu$ defines a map 
          \[\mu^\ast:\AR_d \ \longrightarrow\ \fld\] 
and we can consider 
the quotient $\mathcal{B}$ of $\AR$ under the annihilator of this map. $\mathcal{B}$ is Gorenstein of Krull dimension zero. This quotient is of 
course independent of the precise choice of $\mu^\ast$, and the precise choice is nothing but the normalization of the volume map.

And we can consider any facet $F$ of $\mu$, and observe that if the coefficient of $\mu$ on $F$ is $\mu_F$, then this is reflected in the relative 
normalization with respect to $I_F$. And while it is a bit troublesome to explain the value of $\vo(\x_F)$ in $\mathcal{B}$, we have a direct 
explanation within $I_F$, as $I_F$ is simply isomorphic to the polynomial ring $\fld[\mbf{x}_{i,i\in F}]$. The value of 
 $\vo(\x_F)$ is therefore naturally subject to 
\[\vo(\Delta(\theta_{i,j})_F \mbf{x}_F)=\mu_F.\]
In other words, relative KM normalization recovers the natural normalization for cycles. 
 
\subsection{Lattice polytopes}

Consider next an IDP lattice polytope $P$ of dimension $d-1$, that is, a lattice polytope such that the semigroup algebra $\fld[K_P]$ is 
standard graded (another word for generated in degree one). Here, $K_P$ is the cone 
over $P\times \{1\}$ in $\Lambda \times \mathbb{Z}$, where $\Lambda$ is the ambient lattice for $P$.

In \cite{APPS}, we studied the volume map of such semigroup algebras, and argued that a certain natural choice can be made. We argue 
here that it agrees with KM normalization. 

Specifically, we are interested in the module $\fld[K_P, \partial K_P]$ that is obtained as the kernel of the map 
\[\fld[K_P]\ \longrightarrow\ \fld[\partial K_P].\]

The socle of the Artinian reduction $\AR$ of $\fld[K_P, \partial K_P]$ is concentrated in degree $d$, and it is natural and imperative to understand the map 
\[\AR_d\ \longrightarrow\ \fld\]
And while the normalization we gave in \cite{APPS} fell a bit out of thin air, we reconstruct it here. 

The key observation, \textbf{to start with}, is that if $P$ is a unimodular lattice $(d-1)$-simplex, that is, a lattice simplex such that the vertices of $P\times \{1\}$ generate $K_P\cap (\Lambda \times \mathbb{Z})$, the semigroup algebra is but the face ring of a simplex, that is, it is the polynomial ring generated by $d$ generators. And
$\fld[K_P, \partial K_P]$, in this case, is the ideal generated by $\x_P$, the product of these generators. We fully understand the normalization thanks to relative KM normalization.

\textbf{Simple case.} If $P$ has a unimodular boundary facet $\tau$ in $\partial P$, then we obtain the desired normalization by matching the face ring picture and setting 
\begin{equation}\label{eq:normie}
	1\ =\ \sum_{p\in (P\setminus\tau)\cap\Lambda} \vo(\x_p\x_\tau) \Delta(\Theta_{|\tau,p}) 
\end{equation}
where $\Theta=(\theta_{i,j})$ is the matrix of coefficients in the linear system of parameters. 

\textbf{General case.} In general, consider a flag $(\tau_i)$ of faces of $P$ such that $\tau_d=P$ and such that $\tau_i$ is a facet of $\tau_{i+1}$. We say a set $\sigma=\{\sigma_0,\ldots,\sigma_d\}\subset(P\cap\Lambda)^{d+1}$ without repetitions is \Defn{coherent} with $(\tau_i)$ if it intersects $\tau_i$ in a set of cardinality $i+1$. We then normalize by setting
\begin{equation}\label{eq:norm2}
	1\ =\ \sum_{\sigma \text{ coherent with } (\tau_i)} \vo(\x_\sigma) \Delta(\Theta_{|\sigma}) 
\end{equation}

We claim that this is consistent with KM normalization. Of course, at first, we do not have a Gorenstein ring. But 
we can create one: If we consider the complex of lattice polytopes $P \cup (v\ast \partial P)$, we have a 
lattice sphere $\Sigma$, and the resulting algebra is naturally Gorenstein, obtained by gluing together 
the semigroup algebras of the individual lattice polytopes, is Gorenstein \cite{BBR}.

If $\Sigma$ has a unimodular facet, then we apply the relative normalization above to that facet. More precisely:

If $P$ has a boundary consisting of unimodular simplices, then we consider the cone over $P$ as a polytope; it follows that 
\[(v\ast \partial P)\ \cup\  P\]
is a lattice complex (again, in the sense of \cite{BBR}), and that it defines a Gorenstein ring, meaning 
that the volume on $P$ extends uniquely to a volume on $(v\ast \partial P) \cup P$. For the simplices $F$ 
of $v\ast \partial P$ the volume is uniquely determined and given by $\det^{-1}(\Theta_{|F})$; the normalization 
in the simple case, that is, Identity \eqref{eq:normie} then follows by the following identity, see \cite[Proposition~11.1]{PP}.

\begin{lemma}\label{descentlemma}
Consider any Gorenstein ring $T$ of socle degree $d$ and Krull dimension $d$. Then in the generic Artinian reduction $\AR$ of $T$, and any homogeneous polynomial $m$ of degree $d-1$, we have
\[\sum_{x_i, i\in I} \Delta(\theta_{|J,i}) \vo(mx_i)\ =\ 0\]
where $I$ indexes the indeterminates of the polynomial ring and $J$ is any subset of $I$ of size $d-1$. 
\end{lemma}

For an arbitrary IDP lattice  polytope $P$, we argue in a similar fashion. Unfortunately, unless $P$ has a 
unimodular facet, \[(v\ast \partial P)\ \cup\  P\] does not have a unimodular facet.
However, we can use the fact that the faces of $v\ast \partial P$ are decidedly more simplicial than $P$: they are cones over polytopes of codimension one. 
 
Let us spin this idea further: $P$ is a $d$-dimensional lattice polytope, and we consideran additional set of vertices
$V=\{ v_1,\cdots, v_{d} \}$.
 
Consider the polyhedral complex 
\[\bigcup_{W\subset V} (\bigvee_{v\in W} v) \ast P^{(d-\# W)}\]
where $\bigvee_{v\in W} v$ denotes the free join over the vertices in $W$ and $P^{(d-\# W)}$ consists of the faces of $P$ of dimension $(d-\# W)$ or less.
 
This complex is grouped into polytopes of different kinds, depending on the size of $W$. For $W=V$, we obtain a unimodular simplex.
 
We observe that the algebra generated by this complex is Gorenstein, and that the KM normalization of the simplex $\bigvee_{v\in V} v$ is consistent with the observed rule, as it is just a unimodular simplex.
 
Assuming then that we have the normalization formula for $W$ of cardinality at least $s$, Lemma~\ref{descentlemma} gives the normalization for cardinality $s-1$. This recovers Identity \eqref{eq:norm2}, and the consistency of the volume normalization. Hence, the relative formalism has allowed us to recover the following

\begin{corollary}
	Consider $\Sigma$ a sphere built out of IDP lattice polytopes. Then the algebra it generates is Gorenstein, and at any facet $P$ of $\Sigma$ satisfies Identity \eqref{eq:norm2}.
\end{corollary}

\section* {Acknowledgements}   \label{sec!acknowledgements}
S. A. P. and  V. P.  thank Frank-Olaf Schreyer for useful discussions.
They also  benefited from
experiments with the computer algebra program Macaulay2~\cite{GS}. 
K. A. and V. P. were financially supported by Horizon Europe ERC Grant number: 101045750 /
Project acronym: HodgeGeoComb.

\section*{Appendix:  A brief survey of Kustin-Miller unprojection} \label{app!appentixkmunprojection}

The present appendix is a brief survey of Kustin-Miller unprojection.

\subsection{Introduction}  \label{sec!introduction}

The structure of Gorenstein rings of codimension $\leq 3$  is well understood, (see e.g. \cite[ Corollary~ 21.20]{Eis} 
and \cite[ Theorem~ 3.4.1]{BH}). More precisely, a codimension~$1$ Gorenstein ring is a hypersurface, a  
codimension~$2$ Gorenstein ring is a complete intersection and a  codimension~$3$ Gorenstein ring is Pfaffian.
In the $1980$s,  A. Kustin and M. Miller \cite{KM3, KM1,KM2, KM,KM4,KM5} during their investigation of the 
structure of Gorenstein rings of codimension greater or equal than~$4$ introduced a method that constructs 
Gorenstein rings of more complicated form starting from simpler ones. This  method is known as Kustin-Miller unprojection.

Some years later, M. Reid who was also interested to understand the structure of Gorenstein rings in low codimension,  
motivated by some problems coming from algebraic geometry, reinterpreted the theory of A. Kustin and M. Miller by 
giving his formulation. S.~Papadakis and M.~ Reid  \cite{P, PR} developed further the theory of unprojection. Besides 
of Kustin-Miller unprojection, known also as type I unprojection,  there are several other kinds of 
unprojection~\cite{P3,P6,P4,P5, R1, TR}. However,  in this survey  we will only focus  on the Kustin-Miller unprojection.

The assumptions of Kustin-Miller unprojection are that $I,J\subset R$ are two ideals of a positively graded 
Gorenstein ring $R$ such that $R/I, R/J$ are Gorenstein, $I\subset~J$ and $\dim \  R/I- ~\dim \  R/J= 1$. Then, 
there exists an $R/I$-module homomorphism $\phi: ~ J\rightarrow ~ R/I$ such that the $R/I$-module  
 $\operatorname{Hom}_{R/I}(J,R/I)$ is generated by $\phi$ and the inclusion $i\colon J\rightarrow~ R/I$. 
According to M. Reid, $\operatorname{Hom}_{R/I}(J,R/I)$ contains all the important information for the 
construction of unprojection ring. Under these assumptions, M. Reid and S. Papadakis~(\cite[ Theorem~ 1.5]{PR}), proved that the ring of unprojection is Gorenstein. 

Kustin-Miller unprojection can be used several times, inductively, for the construction of Gorenstein rings of higher 
codimension. This process is called parallel Kustin-Miller unprojection, and was developed by J. Neves and S.~Papadakis \cite{NP2}.

Unprojection theory can be considered as the algebraic language for the study of certain constructions 
in algebraic geometry.  The range of its applications is wide. In explicit algebraic geometry, it is used in the study and the 
construction of some interesting geometric objects such as surfaces of general type,  K3 surfaces, Fano $3$-folds 
and Calabi-Yau $3$-folds ~\cite{AL, ABR, BG,NP1,  NP2, PV, PV1}. The graded ring database ~\cite{BR} contains lists 
of graded rings of such objects whose existence is conjectured and may possibly be proved via unprojection. Moreover, 
it is used in the further development of the  Minimal Model Program by providing an effective way to study explicitly varieties 
and morphisms between them  ~\cite{BKR,BGR,CPR,CM}. Also, it appears in algebraic combinatorics, in the study of 
Stanley-Reisner ideal of cyclic polytopes and stellar subdivisions of Gorenstein complexes \cite{BPBPBP1, BP1, BP3, BP4}.

Kustin-Miller unprojection is related to liaison theory (also known as linkage theory). For more details we refer to 
\cite{KM4}  and \cite[Section~2.6]{Ketal}.

The further development of unprojection theory is an active area of research. There are many recent 
contributions \cite{CL1, CL2, PV, PV1, TR} on foundational questions, computational questions and
applications of unprojection theory.  In the current research work, we use unprojection theory and 
especially Kustin-Miller unprojection, to define the normalization of the volume map.

The survey is organised as follows.
In Subsection~\ref{subs!kmunproj433}  we recall the main principles of unprojection as formulated by M. Reid and S. Papadakis.  In 
Subsections \ref{subs!citoci}, \ref{subs!tmjr},  we focus on two kinds of Kustin-Miller unprojection which lead to the 
construction of Gorenstein rings of codimension $3$ and $4$ respectively. In Subsection \ref{subs!parkmunpr}, we refer 
briefly to the parallel Kustin-Miller unprojection of J. Neves and S. Papadakis and its applications. In Subsection~\ref{sec!kmorig}, we 
present the original construction of A. Kustin and M. Miller known as the Kustin-Miller complex construction. In 
Subsection~\ref{subs!m2pack}, we discuss the Macaulay2 package "KustinMiller", developed by J. B\"{o}hm and S. Papadakis, which
implements the Kustin-Miller complex construction on the the computer algebra system Macaulay2 \cite{GS}.

\subsection{Kustin-Miller unprojection}  \label{subs!kmunproj433}
In this subsection, we follow \cite{PR}. Assume $R$ is a positively graded Gorenstein ring. Let $I,J$ be homogeneous ideals 
of $R$ such that $R/I$ and  $R/J$ are Gorenstein, $I\subset J$ and 
$\dim \  R/I- ~\dim \  R/J=~1$. 
By duality theory,  applying the functor $\operatorname{Hom}_{R/I}(-, {R/I})$ to the exact sequence
\begin{equation}  \label{eqn!no1}
     0\rightarrow \  {J}\rightarrow  R/I
           \rightarrow {R/J}  \rightarrow 0  
\end{equation}
we get the following exact sequence
\begin{equation}   \label{eqn!no2}
     0\rightarrow \  {R/I}\rightarrow \operatorname{Hom}_{R/I}(J, {R/I})
           \rightarrow {R/J}  \rightarrow 0  
\end{equation}
where the second nonzero map corresponds to the Poincar\'e residue map.\\
Using that $R/I$ and $R/J$ are cyclic, we conclude that the $R/I$-module $\operatorname{Hom}_{R/I}(J, R/I)$ is generated by 
two elements. Hence, there exists an $R/I$-module homomorphism $\phi: J\rightarrow~R/I$ such that $\phi$ together with the inclusion 
map $i: J\rightarrow R/I$ generate $\operatorname{Hom}_{R/I}(J, R/I)$ as $R/I~$-module.

\begin{definition}(\textit{Reid})    \label{def!unprr}
Assume $T$ is a new variable.  We define as \textit{Kustin-Miller unprojection ring} of the pair $J\subset R/I$ the quotient
\[ \operatorname{Unpr}(J,R/I) =\frac{R[T]}{I+(Tr-\phi(r): r \in J)}\]
\end{definition}

\begin{theorem}(Papadakis-Reid)
The ring $\operatorname{Unpr}(J,R/I)$ is Gorenstein.
\end{theorem}

\begin{remark}
The  Kustin-Miller unprojection ring $\operatorname{Unpr}(J,R/I)$  does not depend (up to isomorphism) on the choice of the map $\phi$.
\end{remark}

The following example is the simplest example of Kustin-Miller unprojection.  Although the algebra behind this example 
is quite simple, it has found many applications in birational geometry (\cite[ Applications~ 2.3.]{PR}).

\begin{example} (\textit{Reid's Ax-By argument}) \label{exam!milesarg}
Let $A,B\in R= \fld[x,y,z,w]$.  Assume that 
\[X=V(Ax-~By)\subset~  \mathbb{P}^3\]
 is an irreducible cubic surface and 
\[D=V(x,y)\] is a codimension $1$ subscheme contained in $X$. 

Denote by $I_X, I_D$ the ideals correspond to $X$ and $D$ respectively. We can easily check that the ideals $I_X$ 
and $I_D$ of $R$ satisfy the conditions of Kustin-Miller unprojection. 

Then, $\operatorname{Hom}_{R/I_X}(I_D,R/I_X)$ is generated as $R/I_X$-module by the inclusion map $i$ and $\phi$.
The $R/I_X$-module homomorhism  $\phi\colon I_D\rightarrow  R/I_X$  is the unique $R/I_X$-module homomorphism 
such that  $\phi(x)=B+I_X$ and  $\phi(y)=A+I_X$.

The Kustin-Miller unprojection ring of the pair $I_D\subset R/I_X$ is
\[ \operatorname{Unpr}(I_D,R/I_X)= \frac{ R[T]}{ I_X+(Tx-B, Ty-A)}.\]
Denote by $Y= V(I_X+(Tx-B, Ty-A))\subset  \mathbb{P}^4$. We note that $Y$ is a codimension $2$ complete intersection.

There is a natural projection 
\[\pi: Y \DashedArrow[] X  \, \, with \, \,  [x,y,z,w,T]\mapsto [x,y,z,w] 
\]
with birational inverse the rational map (induced by the module homomorphism $\phi$)
\[\pi^{-1}: X \DashedArrow[] Y  \, \, with \, \,  [x,y,z,w]\mapsto [x,y,z,w,T=\frac{A}y=\frac{B}x] 
\]
We have the following cases:

\begin{itemize}
\item (General Case) If $X$ is smooth then $\pi^{-1}$ is a regular map, the usual Castelnuovo blow-down of the $(-1)$-line $D$.
\item (Special Case) Assume that  $V(A,B) \cap V(x,y)\neq \varnothing $. 
For an explicit example, taken from \cite[p.~21]{P1},
assume $X = V(xz(w+x)-y(z+y)w)$.
Then  $V(A,B) \cap V(x,y)=\{P_1,P_2\}$, where $P_1, P_2$ are $A_1$ singularities 
of $X$.  Denote by $Z$ the graph of $\pi$.  Then $\pi^{-1}$ factorizes as follows:
\[
\begin{tikzcd}[column sep=small]
& Z \arrow[dl] \arrow[dr] & \\
X \arrow[rr, dashed, "\pi^{-1}"]& & Y
\end{tikzcd}
\]
The map $Z\rightarrow X$ is the blowup of $X$ at the points $P_1, P_2$ and the map  $Z\rightarrow Y$ is the blowdown of the strict transform of $D$.
\end{itemize}

\end{example}

\begin{remark}
In the above example, $\pi^{-1}$ is the inverse of a projection, which explains the name unprojection.
\end{remark}

\subsection{Unprojection of a complete intersection inside a complete intersection}\label{subs!citoci}

Assume $R$ is a Gorenstein ring and $I$, $J$ are complete intersection ideals of codimension $r$ and $r+1$ respectively 
with the property $I\subset J$. Following \cite[Section~ 4]{P}, we describe the unprojection ring $\operatorname{Unpr}(J,R/I)$.

The ideals $I,J$ are generated by the regular sequences  $\{u_1, \dots , u_r\}$ and $\{v_1, \dots ,v_{r+1}\}$ respectively. That is,
\[
  I = (u_1, \dots , u_r), \quad J = (v_1, \dots ,v_{r+1}).
\]
Since $I \subset J$, there exists an  $r \times (r+1)$ matrix $A$ with
 \[\begin{pmatrix} 
         u_1 \\ \vdots \\ u_r  
   \end{pmatrix}   = A
     \begin{pmatrix} 
         v_1 \\ \vdots \\ v_{r+1}  
   \end{pmatrix}.  \]

Denote by $A_i$ the $r \times r$ submatrix of $A$ obtained by removing the i-th column of A. We 
set for $1\leq i\leq r+1$, $h_i$ to be the determinant of the matrix $A_i$.

\begin{theorem} (\cite[Theorem~4.3]{P})
Denote by 
\[\phi: J\rightarrow R/I\]
the map such that
\[\, \, \, \, \, \, \, \,\, \, \, \, \, \, \, \, \phi(v_i)=(-1)^{i+1} h_i, \,  \, \, \text{for} \,  \, 1\leq i\leq r+1.\]
The $R/I$-module ${\operatorname{Hom}}_{R/I}(J, R/I)$ is generated by $\phi$ and the canonical inclusion $i\colon J/I\rightarrow~R/I$.
Moreover the unprojection ring,
\[ \operatorname{Unpr}(J,R/I) =\frac{R[T]}{I+(Tv_i-(-1)^{i+1} h_i)}\]
is Gorenstein.
\end{theorem}

We finish this subsection by describing the unprojection ring in a specific example of a complete intersection inside a complete intersection.

\begin{example}
Let $R=\fld[c_i, d_i, x_i]$, where $1\leq i\leq 3$, be the 
standard graded polynomial ring in $9$ variables over a field $\fld$.  Consider the ideals
\[
         I=( c_1x_1+c_2x_2+c_3x_3,  d_1x_1+d_2x_2+d_3x_3),  \quad \quad  J=(x_1, x_2, x_3)
\] 
of $R$. The ideals $I,J$ of $R$ satisfy the assumptions of Kustin-Miller unprojection.

\newpage

We consider the $2\times 3$ matrix 

\[ 
A =  \begin{pmatrix} 
        c_1 &  c_2 &  c_3\\
         d_1 &  d_2  & d_3
   \end{pmatrix}
\]
and, for $1 \leq i \leq 3$, we denote by $A_i$ the $2\times 2$ submatrix of $A$  obtained by removing the $i$-th column of A.
For  $1 \leq i \leq 3$, denote by $|A_i|$ the determinant of $A_i$.
We set
\[h_1= |A_1|= c_2d_3-c_3d_2,\]
\[h_2= |A_2|= c_1d_3-c_3d_1,\]
\[h_3= |A_3|= c_1d_2-c_2d_1,\]
Then, by \cite[Theorem~4.3]{P}, the unprojection ring  of the pair  $J\subset R/I$ is
\[ \operatorname{Unpr}(J,R/I) =\frac{\fld[c_1, c_2, c_3, d_1, d_2, d_3, x_1,x_2,x_3, T]}{I+(Tx_1-h_1, Tx_2-(-h_2), Tx_3-h_3)}.\]

We remark that $ \operatorname{Unpr}(J,R/I)$ is not a complete intersection.

\end{example}

\subsection{Tom and Jerry unprojections}\label{subs!tmjr}
Tom and Jerry families, defined and named by M.~Reid,  are two different families of unprojection which are 
used for the construction of a codimension $4$ Gorenstein ideal with $9$ equations and $16$ first syzygies starting from a codimension $3$ Gorenstein ideal. 
S. Papadakis \cite{P} gave an explicit presentation of the unprojection ring of these families.

Before we introduce Tom and Jerry families of unprojection let us remind some preliminary notions.

\begin{definition}
A \textit{skewsymmetric} matrix  $M=[m_ {ij}]$, $1\leq i, j\leq n$ with entries in a commutative ring $R$ is
 an $n\times n$ matrix such that $m_{i,i}=0$ and $m_{ij}=-m_{ji}$.
\end{definition}

For example, a $5\times 5$ skewsymmetric matrix $M$ is of the form
\[ \begin{pmatrix} 
0 & m_{12} & m_{13} & m_{14}& m_{15}\\
-m_{12} & 0  & m_{23} & m_{24}& m_{25}\\
-m_{13} & -m_{23}  & 0 & m_{34}&  m_{35}\\
-m_{14} & -m_{24}  & -m_{34} & 0&  m_{45}\\
-m_{15} & -m_{25}  & -m_{35} & -m_{45}&  0
\end{pmatrix}\]

\begin{definition}
Let $M=[m_ {ij}]$, $1\leq i, j\leq n$ with entries in a commutative ring $R$ be an $n\times n$ skewsymmetric 
matrix. Denote by $M_i$ the skewsymmetric submatrix of M obtained
by deleting the i-th row and i-th column of M and by $ I_{n/2}$  the $n/2 \times n/2$ identity matrix.
\begin{enumerate}
\item If $n$ is even, we call \textit{Pfaffian} of the matrix $M$ and denote it by $\operatorname{Pf}(M)$, the unique polynomial with the properties
\[ (\operatorname{Pf}(M))^2=\det M\]
and
\[\operatorname{Pf}
 ( \begin{pmatrix} 
0 &  I_{n/2}  \\
-I_{n/2} & 0 
\end{pmatrix})=1\] 
\item If $n$ is odd, we call \textit{ Pfaffians} of the matrix $M$ the set  $\{ \operatorname{Pf}(M_1),
\operatorname{Pf}(M_2),\dots, \operatorname{Pf}(M_n)\}$.
\end{enumerate} 
\end{definition}

\begin{example}
If $n=2,$

\[
\operatorname{Pf} (M) = m_{12}.
\]

If $n=5,$

\[
\operatorname{Pf} (M) =  \{ \operatorname{Pf}(M_1),\operatorname{Pf}(M_2),\dots , \operatorname{Pf}(M_5)\}
\]

\newpage 

where

\[
 \operatorname{Pf}(M_1) = m_{23}m_{45} -m_{24}m_{35}+m_{25}m_{34}, 
\]
\[
 \operatorname{Pf}(M_2) = m_{13}m_{45} -m_{14}m_{35}+m_{15}m_{34} ,
\]
\[
 \operatorname{Pf}(M_3) = m_{12}m_{45} -m_{14}m_{25}+m_{15}m_{24}, 
\]
\[
  \operatorname{Pf}(M_4) = m_{12}m_{35} -m_{13}m_{25}+m_{15}m_{23},
\]
\[
 \operatorname{Pf}(M_5) = m_{12}m_{34} -m_{13}m_{24}+m_{14}m_{23}.
\]
\end{example}

Let $M$ be a $5\times 5$ skewsymmetric matrix. Fix a codimension $4$ complete intersection ideal $J$. We set the following question:

\begin{question}\label{qn!no1}
What conditions should be satisfied by the entries of $M$ such that the codimension $3$ ideal generated by the Pfaffians of $M$ is contained in $J$?
\end{question} 

Two different answers in this question are given by Tom and Jerry. According to our knowledge it is still 
an open question  if Tom and Jerry are the only answers to the Question~\ref{qn!no1}  (\cite[Problem~8.2]{R1}).

\begin{definition}
\begin{enumerate}
\item Assume   $1\leq i\leq 5$. The matrix $M$ is called  $Tom_i$  in $J$  if after we delete the i-th row 
and i-th column of $M$ the remaining entries are elements of the codimension $4$ ideal $J$.

\item  Assume  $1\leq i < j \leq 5$. The matrix $M$ is called  $Jerry_{ij}$ in $J$ if  all the entries of $M$ that 
belong to the i-th row or i-th column or j-th row or j-th column are elements of $J$, while there is no restriction for the remaining entries of $M$.

\end{enumerate}
\end{definition}

\begin{example}\label{exp!no2}
We work over the  polynomial ring   $ R=\fld[y_k, z_k, m_{ij}^k]$  where $1\leq k \leq 4$, $2\leq i < j \leq 5$. 
Assume $J=(z_1,z_2,z_3,z_4)$ is a codimension $4$ complete intersection ideal of $R$.
\begin{enumerate}
\item The matrix
\[
M=  \begin{pmatrix} 
0 & y_{1} & y_{2} & y_{3}& y_{4}\\
-y_{1} & 0  & m_{23} & m_{24}& m_{25}\\
-y_{2} & -m_{23}  & 0 & m_{34}&  m_{35}\\
-y_{3} & -m_{24}  & -m_{34} & 0&  m_{45}\\
-y_{4} & -m_{25}  & -m_{35} & -m_{45}&  0
\end{pmatrix},
\]
where  
\[
 m_{ij}= \sum_{k=1}^{4} m_{ij}^k z_k.
\]
 is an example of a $Tom_1$ matrix in $J$.
\item The matrix
\[
M=  \begin{pmatrix} 
0 & m_{12} & m_{13} & m_{14}& m_{15}\\
-m_{12} & 0  & m_{23} & m_{24}& m_{25}\\
-m_{13} & -m_{23}  & 0 &y_1&  y_2\\
-m_{14} & -m_{24}  & -y_1 & 0&  y_3\\
-m_{15} & -m_{25}  & -y_2 & -y_3&  0
\end{pmatrix},
\]
where  
\[
 m_{ij}= \sum_{k=1}^{4} m_{ij}^k z_k.
\]
 is an example of a $Jerry_{12}$ matrix in $J$.
\end{enumerate}
\end{example}

We finish this subsection by recalling the main ideas of Papadakis' calculation about Tom (\cite[Section 3.3]{P}).

\,

{\bf Papadakis' Calculation about Tom}
We work over the  polynomial ring   $ R=\fld[y_k, z_k, m_{ij}^k]$  where $1\leq k \leq 4$, $2\leq i < j \leq 5$.
Consider the $Tom_1$ matrix $M$ in the ideal $J=(z_1, z_2, z_3, z_4)$ of the first part of Example  \ref{exp!no2}. Denote 
by $ I$ be the ideal generated by the Pfaffians $\operatorname{Pf}(M_1),\operatorname{Pf}(M_2),\dots , \operatorname{Pf}(M_5)$  
of $M$. It is easy to see that $I \subset J.$

Using that  $\operatorname{Pf}(M_2),\dots , \operatorname{Pf}(M_5)$ are linear 
in  $z_1, z_2, z_3, z_4$, there exists a unique  $4\times 4$ matrix $Q$ such that
\[
 \begin{pmatrix} 
\operatorname{Pf}(M_2) \\ \operatorname{Pf}(M_3)  \\ \operatorname{Pf}(M_4)  \\ \operatorname{Pf}(M_5) 
\end{pmatrix} =   Q  \begin{pmatrix} 
z_{1} \\ z_{2} \\  z_{3} \\ z_{4}
\end{pmatrix}.
\]

For  $i = 1,\dots, 4$, let  $ H_i$  be the  $1\times 4$ matrix whose i-th entry is equal to $ (-1)^{i+1}$ times the 
determinant of the submatrix of $Q_{i}$ obtained by $Q$ deleting the i-th column.  
We fix \, $1\leq j\leq 4$. Using that for all $i, j$,  

\begin{center}
$y_{i}  H_{j} = y_{j}  H_{i}.$
\end{center}
we can define four polynomials $g_1, g_2, g_3, g_4$ such that
\begin{center}
$(g_1, g_2, g_3, g_4) =   H_{j}/y_{j}$.
\end{center}
We note that this definition does not depend on the choice of $j$.

Denote by $\phi$ the map which is defined by
\begin{center}
$\phi \colon J \rightarrow R/I$
\end{center}

\begin{center}
 $ z_i \mapsto  g_i$.
\end{center} 
Then, $\operatorname{Hom}_{R/I}(J,R/I)$ is generated as $R/I$-module by the inclusion map $i$ and $\phi$.
Denote by 
\[I_{un}= I+( Tz_1-g_1, Tz_2-g_2,   Tz_3-g_3,   Tz_4-g_4).\]
The unprojection ring of the pair  $J/I\subset R/I$ 
\[ \operatorname{Unpr}(J,R/I) =\frac{\fld[y_k, z_k, m_{ij}^k,T]}{I_{un}}\]
is Gorenstein and the codimension of the ideal $I_{un}$ is equal to $4$.

\begin{remark}
According to our knowledge it is  still an open question  if every codimension~$4$ Gorenstein ring 
with a $9\times 16$  resolution comes from Tom and Jerry unprojections (\cite[Problem~8.3]{R1}).
\end{remark}

\subsection{Parallel Kustin-Miller unprojection} \label{subs!parkmunpr}
 In Algebraic Geometry, especially for applications it is necessary to produce Gorenstein rings of higher 
codimension. For this aim,  Kustin-Miller unprojection can be used more than one time in order to produce 
Gorenstein rings of arbitrary codimension whose properties are controlled by just a few equations as a number 
of new unprojection variables are adjoined.

In this direction, J. Neves and S. Papadakis \cite{NP2} developed a theory which is called parallel Kustin-Miller 
unprojection. The initial data for parallel Kustin-Miller unprojection is a Gorenstein graded ring $R$ and a finite 
set of codimension $1$ ideals $\{J_1,\dots,J_n\}$ such that the quotients $R/J_i$ are Gorenstein and satisfy 
some extra mild assumptions. The unprojection ring that is obtained by this process is defined as the 
quotient of $R[T_1,\dots, T_n]$,  where $T_1,\dots T_n$ are new variables, by an ideal of simple form obtained from
the initial data. Moreover, the new ring is Gorenstein(\cite[Theorem~ 2.3]{NP2}). 

Parallel Kustin-Miller unprojection has found many applications in the construction of new interesting algebraic varieties 
\cite{NP1,NP2,NPI,PV,PV1}. For more details related to parallel Kustin-Miller unprojection we refer the reader to \cite{NP2, PV}.

\subsection{Kustin-Miller complex construction}\label{sec!kmorig}
Kustin-Miller complex construction was introduced by A. Kustin and M. Miller \cite{KM} during 
their efforts to find a structure theorem of Gorenstein rings of codimension $\geq 4$. Given a pair of projective  
resolutions of Gorenstein rings that satisfy certain properties, this 
construction produces a new Gorenstein ring and its resolution. 

Kustin-Miller complex construction has found many applications in algebraic geometry, for example in  the construction 
of some interesting geometric objects such as Campedelli surfaces of general type and Calabi-Yau $3$-folds of 
high codimension \cite{NP1,NP2}. In these cases, the numerical invariants 
of the varieties produced via Kustin-Miller unprojection are controlled by the Kustin-Miller complex construction. Moreover, it has 
found many applications in algebraic combinatorics, for example 
 in the study of face rings of cyclic polytopes and stellar subdivisions of Gorenstein 
complexes \cite{BPBPBP1,BP1, BP3, BP4}. 

J. B\"{o}hm and S. Papadakis \cite{BPBPBP6,BPBPBP5}  developed the Macaulay2 \cite{GS} package "KustinMiller" 
which implements the Kustin-Miller complex construction.

In this subsection, we describe the Kustin-Miller complex construction, following  \cite{BPBPBP1, BPBPBP5, KM}.

\,

We begin by recalling the assumptions of Kustin-Miller unprojection of Subsection~\ref{subs!kmunproj433}.

Assume $R$ is a positively graded polynomial ring over a field. Let $I,J$ be homogeneous ideals of $R$ such that $R/I$ and  $R/J$ are Gorenstein, $I\subset J$ and 
$\dim \  R/I- ~\dim \  R/J=~1$.

\textit{ Kustin-Miller complex construction} is the construction of the graded free resolution of the Kustin-Miller unprojection 
ring $\operatorname{Unpr}(J,R/I)$ (Definition \ref{def!unprr}) from  graded free resolutions of $R/I$ and $R/J$ as $R$-modules. 

Using that $R/I$, $R/J$ are Gorenstein rings, by \cite[Proposition~3.6.11]{BH} there 
are integers $k_1, k_2$ such that $\omega_{R/I}=R/I(k_1)$ and $\omega_{R/J}=R/J(k_2)$. Assume 
that $k_1>k_2$. The assumptions of Kustin-Miller unprojection are satisfied. Hence, we fix a graded 
homomorphism $\phi\colon J\rightarrow R/I$ of degree $k_1-k_2$ such that  $\operatorname{Hom}_{R/I}(J, R/I)$ is generated 
as an $R/I$-module by $\phi$ and the canonical inclusion $i$.  The Kustin-Miller 
unprojection ring of the pair $J\subset R/I$, $\operatorname{Unpr}(J,R/I)$ defined by the $\phi$ (Definition \ref{def!unprr}), 
where $T$ is a new variable of degree $k_1-k_2$, is a positively graded algebra. 

We now describe the construction given by A. Kustin and  M. Miller \cite{KM} of the graded free resolusion 
of the Kustin-Miller unprojection ring $\operatorname{Unpr}(J,R/I)$ (Definition \ref{def!unprr}) from the graded free resolutions of $R/I$ and $R/J$.

Denote by $g$ the codimension of the ideal $J$ of $R$. That is, $g= \dim R-  \dim R/J$. 
Let
\[%
\begin{tabular}
[c]{ll}%
$C_{J}:$ & $R/J\leftarrow \theta_{0}=R\overset{\theta_{1}}{\leftarrow}\theta_{1}\overset{\theta_{2}%
}{\leftarrow}\dots\overset{\theta_{g-1}}{\leftarrow}\theta_{g-1}\overset{\theta_{g}%
}{\leftarrow}\theta_{g}=R(-k_1-\eta)\leftarrow0$\\
$C_{I}:$ & $R/I\leftarrow B_{0}=R\overset{b_{1}}{\leftarrow}B_{1}\overset{b_{2}%
}{\leftarrow}\dots\overset{b_{g-1}}{\leftarrow}B_{g-1}=R(-k_2-\eta)\leftarrow0$%
\end{tabular}
\]
be minimal graded free resolutions  of $R/J$ and $R/I$ respectiely as $R$-modules, where $\eta$ is the sum 
of the degrees of the variables of $R$. Due to Gorenstein property, $C_{I}$ and $C_{J}$ are self-dual (\cite[Corollary~ 21.16]{Eis}).

For an $R$-module $M$, we denote $M^{\prime}:=M\otimes_{R}R[T]$ which is an $R[T]$-module.
Consider the complex%
\[%
\begin{tabular}
[c]{ll}%
$C:$ & $\operatorname{Unpr}(J,R/I)\leftarrow F_{0}\overset{f_{1}}{\leftarrow}F_{1}%
\overset{f_{2}}{\leftarrow}\dots\overset{f_{g-1}}{\leftarrow}F_{g-1}%
\overset{f_{g}}{\leftarrow}F_{g}\leftarrow0$%
\end{tabular}
\]
with the modules, when $g\geq 4$,%
\begin{gather*}%
\begin{tabular}
[c]{lll}%
$F_{0}=B_{0}^{\prime}$, &  & $F_{1}=B_{1}^{\prime}\oplus \theta_{1}^{\prime}%
(k_{2}-k_{1})$%
\end{tabular}
\\%
\begin{tabular}
[c]{ll}%
$F_{i}=B_{i}^{\prime}\oplus \theta_{i}^{\prime}(k_{2}-k_{1})\oplus B_{i-1}^{\prime
}(k_{2}-k_{1})$, & f$\text{or }2\leq i\leq g-2$%
\end{tabular}
\\%
\begin{tabular}
[c]{lll}%
$F_{g-1}=\theta_{g-1}^{\prime}(k_{2}-k_{1})\oplus B_{g-2}^{\prime}(k_{2}-k_{1})$, &
& $F_{g}=B_{g-1}^{\prime}(k_{2}-k_{1}).$%
\end{tabular}
\end{gather*}

If $g=2$, we have%
\begin{gather*}%
\begin{tabular}
[c]{lll}%
$F_{0}=B_{0}^{\prime}$, &  & $F_{1}= \theta_{1}^{\prime}%
(k_{2}-k_{1})$%
\end{tabular}
\\%
\begin{tabular}
[c]{lll}%
$F_{2}=B_{1}^{\prime}(k_{2}-k_{1})$.%
\end{tabular}
\end{gather*}

If $g=3$, we have%
\begin{gather*}%
\begin{tabular}
[c]{lll}%
$F_{0}=B_{0}^{\prime}$, &  & $F_{1}=B_{1}^{\prime}\oplus \theta_{1}^{\prime}%
(k_{2}-k_{1})$%
\end{tabular}
\\%
\begin{tabular}
[c]{lll}%
$F_{2}=\theta_{2}^{\prime}(k_{2}-k_{1})\oplus B_{1}^{\prime}(k_{2}-k_{1})$, &
& $F_{3}=B_{2}^{\prime}(k_{2}-k_{1}).$%
\end{tabular}
\end{gather*}
We will define the differentials of the complex $C$ by specifying chain maps  $\alpha:C_{I}\rightarrow C_{J}$, $\beta:C_{J}\rightarrow
C_{I}[-1]$ and a homotopy map $h:C_{I}%
\rightarrow C_{I}$ which is not required to be a chain map.

Let $t_1$ be the rank of the free $R$-module $A_1$. The self-duality of $C_{J}$ implies that $t_1$ is also 
the rank of  the free R-module $\theta_{g-1}$.

Fix $R$-module bases $e_1,\dots,e_{t_1}$ and ${\hat{e}}_1,\dots,{\hat{e}}_{t_1}$ of $\theta_{1}$ and $\theta_{g-1}$ 
respectively. For $1~\leq i\leq t_1$, we define $c_i, \hat{c_i}\in R$ such that
\[\theta_{1}(e_{i})=c_{i} 1_{R}, \, \, \, \, \,  \theta_{g}(1_R)= \sum_{i=1}^{t_{1}} \hat{c_i} \hat{e_i}\]
where by Gorenstein property, $c_{i}, \hat{c_i}\in J$. 

 For $1\leq i\leq t_1$, denote by $l_i,  \hat{l_i}\in R$  lifts in $R$ of $\phi(c_i)$ and $\phi(\hat{c_i})$ 
respectively. For an $R$-module $A$, we set $A^{\ast}= {Hom}_{R}(A, R)$. For an $R$-basis $f_1,\dots ,f_t$,  
denote by $f_1^{\ast},\dots ,f_t^{\ast}$ the dual basis of $A^{\ast}$.

Let ${{\bar{a}}_{g-1}}^{d}\colon \theta_{g-1}^{\ast}\rightarrow R=B_{g-1}^{\ast}$ be the $R$-module 
homomorphism defined by ${{\bar{a}}_{g-1}}^{d}({\hat{e_i}}^{\ast})=\hat{l_i}1_R$ 
for $1\leq i\leq t_1$. Using that $C_{I}, C_{J}$ are self-dual we get that ${{\bar{a}}_{g-1}}^{d}$ 
extends to  ${{\bar{a}}}^{d}\colon C_{J}^{\ast}\rightarrow C_{I}^{\ast} $. 
Denote by ${\bar{a}}\colon C_{I}\rightarrow C_{J}$ its dual.

By \cite{P}, the map $\bar{a}_{0}\colon B_0=R\rightarrow R=\theta_{0}$ is multiplication by an invertible element of $R$,
say $u$. We set $a=\bar{a}/u$.

Next step is to define a chain map $\beta\colon C_{J}\rightarrow C_{I}[-1]$ by 
extending $\beta_{1}\colon \theta_{1}\rightarrow R=B_{0}$  defined by $\beta_{1}(e_i)=-l_i 1_R$.

The map $\beta_{g}\colon \theta_{g}=R\rightarrow R=B_{g-1}$ is multiplication by a nonzero constant. Then, by 
the proof of  (\cite[Theorem 1.4.]{KM}) there is a homotopy map $h\colon C_{I}\rightarrow C_{I}$ 
with $h_0\colon B_0\rightarrow B_0$ and $h_{g-1}\colon B_{g-1}\rightarrow B_{g-1}$ being 
the zero maps and
\[\beta_{i}\alpha_{i}=h_{i-1}b_{i}+b_{i}h_{i},\text{\quad for }1\leq i\leq g.\]

Denote by $r=\operatorname{rank} \,  B_{t}$. Let $I_{t}$ be the identity $r\times r$ matrix. The differential 
maps $f_i\colon F_{i}\rightarrow F_{i-1}$ of the complex $C$ are defined as: 
\begin{gather*}%
\begin{tabular}
[c]{lll}%
$f_{1}=\left(
\begin{array}
[c]{cc}%
b_{1} & \beta_{1}+T\cdot \theta_{1}%
\end{array}
\right)  $, &  & $f_{2}=\left(
\begin{array}
[c]{ccc}%
b_{2} & \beta_{2} & h_{1}+T\cdot I_{1}\\
0 & -\theta_{2} & -\alpha_{1}%
\end{array}
\right)  $%
\end{tabular}
\\%
\begin{tabular}
[c]{ll}%
$f_{i}=\left(
\begin{array}
[c]{ccc}%
b_{i} & \beta_{i} & h_{i-1}+(-1)^{i}T\cdot I_{i-1}\\
0 & -\theta_{i} & -\alpha_{i-1}\\
0 & 0 & b_{i-1}%
\end{array}
\right)  $ & $\text{for }3\leq i\leq g-2$%
\end{tabular}
\\%
\begin{tabular}
[c]{ll}%
$\hspace{-0.05in}f_{g-1}=\left(
\begin{array}
[c]{cc}%
\beta_{g-1} & h_{g-2}+(-1)^{g-1}T\cdot I_{g-2}\\
-\theta_{g-1} & -\alpha_{g-2}\\
0 & b_{g-2}%
\end{array}
\right)  $, & $f_{g}=\left(
\begin{array}
[c]{c}%
-\alpha_{g-1}+(-1)^{g}\frac{1}{\beta_{g}\left(  1\right)  }T\cdot \theta_{g}\\
b_{g-1}%
\end{array}
\right)  $%
\end{tabular}
\end{gather*}

\begin{theorem}\cite{KM}
The chain complex $C$, known as the Kustin-Miller complex construction, is a graded free resolution of $\operatorname{Unpr}(J,R/I)$  as $R[T]$-module.
\end{theorem}

\begin{remark}
In general, the resolution $C$ is not minimal  as indicated, for example, by \cite[Subsection 5.1, Example 3]{BP21}. 
However, it is minimal in some examples coming from algebraic geometry  
\cite{NP1, PV, PV1} and combinatorics  \cite[Theorem 6.1]{BPBPBP1}, \cite{BP1}.
For explicit examples of the Kustin-Miller complex construction, we refer the reader to  \cite[Subsection~6.2]{BPBPBP1}.
\end{remark}

\subsection{The Macaulay2 package KustinMiller} \label{subs!m2pack}

The Macaulay2 package 
"KustinMiller", developed by J. B\"{o}hm and S. Papadakis \cite{BPBPBP5},  is the implementation using Macaulay2~\cite{GS} 
of the Kustin-Miller complex construction. More precisely, the authors present an algorithm which has as input the 
resolutions $C_I$ and $C_J$ of $I$ and $J$ respectively and produces as  output the Kustin-Miller complex $C$ associated to $I$ and $J$
constructed in the previous paragraph.

To use the package, type in Macaulay2  the command line needsPackage"KustinMiller". 

The command  \texttt{kustinMillerComplex} produces the Kustin-Miller complex $C$ from the resolutions $C_I, C_J$. For 
an example, we refer to \cite[Section 4]{BPBPBP5}. Moreover, there are useful  Macaulay2  commands 
such as $\texttt{Hom(J,R\symbol{94}1/I)}$, which determines the unprojection map $\phi$ (Subsection~\ref{subs!kmunproj433}) and  
\texttt{extend} which extends homomorphisms to chain maps. For more details about the package we refer the reader to \cite{BPBPBP6, BPBPBP5}.

\end{document}